\documentclass[12pt,a4paper,reqno]{amsart}
\pdfoutput=1
\usepackage{amssymb,verbatim}
\usepackage{amsfonts}
\usepackage[utf8]{inputenc}
\usepackage[mathscr]{euscript}
\usepackage{enumitem}
\usepackage[hidelinks]{hyperref}
\usepackage{mathrsfs}
\usepackage[margin=0cm,justification=centering]{caption}
\usepackage[export]{adjustbox}
\usepackage{mathtools}
\usepackage{tikz}
\usepackage{etoolbox}
\setenumerate{label=\textup{(\arabic*)},itemsep=3pt,topsep=0pt,leftmargin=1.7em}

\allowdisplaybreaks[4]
\usetikzlibrary{arrows,arrows.meta,calc,positioning}
\usepackage[normalem]{ulem}
\usepackage[
left=1in,
right=1in,
top=1.2in,
bottom=1in,
includefoot,
footskip=0pt,
headsep=20pt
]{geometry}

\newtheorem{theorem}{Theorem}[section]
\newtheorem{lemma}[theorem]{Lemma}
\newtheorem{proposition}[theorem]{Proposition}
\newtheorem{corollary}[theorem]{Corollary}

\theoremstyle{definition}
\newtheorem{definition}[theorem]{Definition}
\newtheorem{example}[theorem]{Example}

\theoremstyle{remark}
\newtheorem{remark}[theorem]{Remark}

\AtEndEnvironment{remark}{\hfill\ensuremath{\Diamond}}
\AtEndEnvironment{example}{\hfill\ensuremath{\Diamond}}

\numberwithin{equation}{section}
\renewcommand{\emptyset}{\varnothing}

\makeatletter
\def\@tocline#1#2#3#4#5#6#7{\relax
	\ifnum #1>\c@tocdepth 
	\else
	\par \addpenalty\@secpenalty\addvspace{#2}%
	\begingroup \hyphenpenalty\@M
	\@ifempty{#4}{%
		\@tempdima\csname r@tocindent\number#1\endcsname\relax
	}{%
		\@tempdima#4\relax
	}%
	\parindent\z@ \leftskip#3\relax \advance\leftskip\@tempdima\relax
	\rightskip\@pnumwidth plus4em \parfillskip-\@pnumwidth
	#5\leavevmode\hskip-\@tempdima
	\ifcase #1
	\or\or \hskip 1em \or \hskip 2em \else \hskip 3em \fi%
	#6 \hskip 0.5em \nobreak\relax
	\dotfill\hbox to\@pnumwidth{\@tocpagenum{#7}}\par
	\nobreak
	\endgroup
	\fi}
\makeatother

\newcommand{\reg}{\mathrm{reg}}
\newcommand{\Id}{\mathrm{id}}

\newcommand{\ol}[1]{\overline{#1}}

\renewcommand{\[}{\begin{equation}}
\renewcommand{\]}{\end{equation}}

\allowdisplaybreaks[4]

\title{\vspace*{-15mm}Relation morphisms of directed graphs}

\author[G.~G.~de Castro]{Gilles G.~de Castro}
\address[G.~G.~de Castro]{Departmento de Matem\'atica, Universidade Federal de Santa Catarina, Florian\'opolis, SC, 88040-900, Brazil}
\email{gilles.castro@ufsc.br}

\author[F.~D'Andrea]{Francesco D'Andrea}
\address[F.~D'Andrea]{Dipartimento di Matematica e Applicazioni ``R.~Caccioppoli'',
	Universit\`a di Napoli Federico II, and INFN Sezione di Napoli, 80126 Napoli, Italy}
\email{francesco.dandrea@unina.it}

\author[P. M.~Hajac]{Piotr M.~Hajac}
\address[P. M.~Hajac]{Instytut Matematyczny, Polska Akademia Nauk, ul. \'Sniadeckich 8, Warszawa, 00-656 Poland}
\email{pmh@impan.pl}

\begin{document}
	
	\begin{abstract}
		Associating graph algebras to directed graphs leads to both covariant and contravariant functors from suitable categories
		of graphs to the category $k$-$\mathsf{Alg}$ of algebras and algebra homomorphisms. 
		As both functors are often used at the same time, finding a new category of graphs that allows
		a ``common denominator'' functor unifying the covariant and
		contravariant constructions is a fundamental problem. Herein, we solve this problem by
		first introducing the relation category of graphs $\mathsf{RG}$, and then determining the concept of admissible graph relations that
		yields a subcategory of $\mathsf{RG}$ admitting a contravariant functor to $k$-$\mathsf{Alg}$
		simultaneously generalizing the aforementioned covariant and contravariant functors. Although we focus on Leavitt path algebras and graph C*-algebras, 
		on the way we unravel functors to $k$-$\mathsf{Alg}$
		given by path algebras, Cohn path algebras and Toeplitz graph C*-algebras from suitable subcategories of~$\mathsf{RG}$. Better still,
		we illustrate relation morphisms of graphs by naturally occurring
		examples, including Cuntz algebras, quantum spheres and quantum balls.
	\end{abstract}
	
	\maketitle
	
	\begingroup
	\quad
	\begin{minipage}{0.87\textwidth}
		\setlength{\parskip}{2pt}
		{\small \tableofcontents}
	\end{minipage}
	\endgroup
	
	\section{Introduction}
	\noindent
	Directed graphs are fundamental objects in combinatorics.
	Through path algebras, they are pivotal in representation theory and
	and the classification of finite-dimensional associative algebras over algebraically closed fields~\cite{ass-06}.
	On the other hand, the introduction of the Cuntz--Krieger relations into path algebras  brought about
	Leavitt path algebras \cite{aasm17} and graph C*-algebras \cite{R05}, which enjoy very tangible and computable K-theory,
	and appear  naturally in noncommutative topology.
	
	The construction of path algebras, Leavitt path algebras, and graph C*-algebras can be considered as a functor from a category of 
	directed graphs to the category of algebras and algebra homomorphisms
	in two different ways: covariant and contravariant. The former was explored in
	\cite{cht21,HT24cov} (cf.~\cite{aasm17,amp07,kr-g09,j-s02}) using a more general concept of morphisms of graphs. The latter uses the standard concept
	of a graph homomorphism, and was spectacularly successful in the work of Stallings~\cite{s-jr83}. Better still, the 
	contravariant induction for admissible subgraphs (also called  quotient graphs)  is ubiquitous, including natural examples in
	noncommutative topology explored by Hong and Szyma\'nski~\cite{HS02,HS08}. The importance and naturality of the contravariant
	functoriality for non-injective graph homomorphisms was explored  in~\cite{k-t06,HT24}.

	The main goal of this paper is to put together the covariant and contravariant functoriality, which we achieve in Theorem~\ref{main}.
	To this end, we introduce the concept of a relation morphism and refine it to an admissible relation morphism.
	The definition of a relation morphism was motivated by Zakrzewski morphisms for groupoids~\cite{z-s90}. As groupoids combine topological
	spaces with topological groups, and natural functors assigning C*-algebras to spaces and groups are contravariant and covariant respectively,
	one needs to resort to relations between groupoids rather then maps to obtain a category of groupoids that admits a functor to the category of C*-algebras
	and *-homomorphisms assigning convolution C*-algebras to groupoids.
	(Since every graph C*-algebra is a groupoid C*-algebra~\cite{p-alt02}, 
	unraveling the precise relationship between relation morphisms and Zakrzewski morphisms is an 
	important issue that is currently investigated.)
	
	While our prime motivation comes from
	mixed-pullback theorems in~\cite{cht21,HT24cov}, 
	where naturally occurring  diagrams require both the covariant and the contravariant functor,
	the need for such a unification is already transparent in studying maps between quantum spheres (Example~\ref{qsph})
	and unital embeddings of Cuntz algebras (Example~\ref{ex:cuntz}). These examples lead us to the following open question
	(Example~\ref{fdec}):
	Is every
	admissible relation morphism decomposable into finitely many admissible morphisms used in the covariant or contravariant
	induction?
	
	After recalling basic notions and establishing notation in Section~\ref{sec:2}, we introduce the relation morphisms of directed graphs and analyze 
	their decompositions in the above sense in Section~\ref{sec:3}. 
	In the subsequent three sections, we unravel the conditions needed to turn the assignment of, respectively, path algebras, 
	Cohn path algebras,
	and Leavitt path algebras to directed graphs into contravariant functors. At each stage, we analyze the aforementioned decomposition
	issue, and check how the new contravariant functor relates to the earlier known and used covariant and contravariant functors.
	Except for the case of path algebras, the new functors generalize the earlier functors subsuming them into the new framework.
	We end the paper by analyzing  in depth a number of examples and counterexamples in Section~\ref{sec:7}.

	\section{Preliminaries}\label{sec:2}
	\subsection{Graphs and paths}

	Throughout the paper,  $k$ is a fixed field, all algebras and algebra homomorphisms are over~$k$, and
	$\mathbb{C}$ denotes the field of complex numbers.
	In this section, we recall some definitions and results from \cite{HT24cov} and \cite{HT24}.
	
	A \emph{directed graph} $E$ (also called a \emph{quiver}) is a quadruple $(E^0,E^1,s_E,t_E)$, where $E^0$ and $E^1$ are disjoint sets and 
	$s_E,\,t_E\colon E^1\to E^0$ are maps.
	We call $E^0$ the set of \emph{vertices}, $E^1$ the set of \emph{edges} (or \emph{arrows}), $s_E$ the source map, and $t_E$ the target (or range) map.
	To simplify terminology, directed graphs will be simply be referred to as graph.
	
	A vertex $v\in E^0$ is called a \emph{sink} if $s_E^{-1}(v)=\emptyset$, it is called an \emph{infinite emitter} if $|s_E^{-1}(v)|=\infty$, and it is called 
	\emph{regular} if it is 
	neither a sink nor an infinite emitter. The set of regular vertices of a graph $E$ is denoted by $\reg(E)$. An edge $e$ with $s_E(e)=t_E(e)$ is called a \emph{loop}.
	We call a vertex \emph{$0$-regular} if it emits exactly one edge and this edge is a loop. The set of $0$-regular vertices of a graph $E$ is denoted by $\reg_0(E)$.
	
	A \emph{path} $p$ in $E$ of length $n\geq 1$ is a sequence $e_1e_2\cdots e_n$ of edges such that
	\[
	t_E(e_i)=s_E(e_{i+1}) \;\;\forall\;1\leq i<n .
	\]
	We think of a vertex $v\in E^0$ as a path of length~$0$. 
	We denote by $ \mathrm{FP}(E)$ the set of all finite paths in $E$.  
	The length of a path $p$ is denoted by~$|p|$.
	We define the source and target maps by $s_{PE}(e_1\cdots e_n):=s_E(e_1)$ and $t_{PE}(e_1\cdots e_n):=t_E(e_n)$,  and  by
	$s_{PE}(v)=t_{PE}(v)=v$. A positive-length path $x$ in $E$ is called a \emph{cycle} if $s_{PE}(x)=t_{PE}(x)$ (a loop is a cycle of length $1$).

	If $x,x'\in  \mathrm{FP}(E)$ satisfy $t_{PE}(x)=s_{PE}(x')$, we denote by $xx'$ their concatenation, and say that $x$ is an \emph{initial subpath} of $xx'$ 
	(Note that if $x$ or $x'$ is a  vertex, we write, respectively, $xx'=x'$ or $xx'=x$.)
	We define a partial order $\preceq$ by letting $x\preceq x'$ if $x$ is an initial subpath of~$x'$.
	We write $x\prec x'$ to mean that $x\preceq x'$ and $x\neq x'$. Finally, we write $x\sim x'$, and say that $x$ and $x'$ are \emph{comparable}, 
	if either $x\preceq x'$ or $x'\preceq x$.

	\subsection{Categories of graphs}
	
	Given two graphs $E$ and $F$, a \emph{path homomorphism} \mbox{$\vartheta:E\to F$}  is a map
	$\vartheta: \mathrm{FP}(E)\to  \mathrm{FP}(F)$ (which we denote by the same symbol) satisfying
	\begin{enumerate}[leftmargin=3em]
		\item\label{en:path1} $\vartheta(E^0)\subseteq F^0$,
		\item\label{en:path2} $s_{PF}\circ\vartheta=\vartheta\circ s_{PE}$, \ $t_{PF}\circ\vartheta=\vartheta\circ t_{PE}$,
		\item\label{en:path3} $\vartheta(xx')=\vartheta(x)\vartheta(x')\;\forall\;x,x'\in  \mathrm{FP}(E)$ such that $t_{PE}(x)=s_{PE}(x')$.
	\end{enumerate}
	It is immediate that the composition of path homomorphisms is again a path homomorphism, and that the identity map is a path homomorphism. 
	We thus arrive at:
	\begin{definition}
		We denote by $\mathsf{PG}$ the category whose objects are all graphs and whose morphisms are path homomorphisms. We denote by $\mathsf{OG}$ the 
		subcategory of $\mathsf{PG}$ whose objects are all graphs and whose morphisms are \emph{length-preserving} path homomorphisms. We call length-preserving
		path homomorphisms \emph{graph homomorphisms}.
	\end{definition}
	
	Note that the  condition \ref{en:path3} determines a path homomorphism by its values on edges and vertices: 
	$\vartheta(e_1\cdots e_n)=\vartheta(e_1)\cdots\vartheta(e_n)$.
	Better still, if $\varphi:E\to F$ is length preserving, then not only it maps vertices to vertices but also edges to edges, so we might equivalently look at $\varphi$ as 
	a pair of maps
	$\varphi^0:E^0\to F^0$ and $\varphi^1:E^1\to F^1$ obtained by restrictions. Vice versa, any such a pair of maps  
	satisfying the obvious adaptation of condition \ref{en:path2} uniquely determines a length-preserving path homomorphism, i.e.\ a graph homomorphism.
	
	\begin{definition}[\cite{HT24}]\label{conadm}
		A morphism $\varphi:E\to F$ in $\mathsf{OG}$ is called
		\begin{enumerate}
			\item 
			\emph{proper} if $|\varphi^{-1}(y)|<\infty,\;\forall\;y\in F^0\cup F^1$;
			\item 
			\emph{target injective}, respectively \emph{target surjective}, if the restriction 
			$$
			t_{PE}:\varphi^{-1}(f)\longrightarrow \varphi^{-1}(t_{PF}(f))
			$$
			of the target map is
			injective, respectively surjective, for all $f\in F^1$;
			\item 
			\emph{target bijective} if it is target injective and target surjective;
			\item 
			\emph{regular} if $\varphi^{-1}(\reg (F))\subseteq \reg (E)$.
		\end{enumerate}
	\end{definition}
	\begin{definition}[\cite{HT24cov}]\label{covadm}
		A morphism $\vartheta:E\to F$ in $\mathsf{PG}$ is called
		\begin{enumerate}
			\item \emph{injective on vertices} if its restriction $\vartheta:E^0\to F^0$ is injective;
			\item \emph{monotone} if, for all $e,e'\in E$, $\vartheta(e)\preceq\vartheta(e')$ implies $e=e'$;
			\item \emph{regular} whenever the following conditions hold:\vspace{2pt}
			\begin{enumerate}
				\item
				For any $v\in\mathrm{reg}(E)\setminus {\rm reg}_0(E)$, we require that:
				\begin{enumerate}[label=(\alph*)]
					\item
					$\vartheta$ restricted to $s_E^{-1}(v)$ be injective;
					\item
					$y\in \vartheta(s_E^{-1}(v))$ if and only if
					\begin{enumerate}[label=(\roman*)]
						\item
						$y=f_1\cdots f_n$ with $n\geq 1$ and $f_1,\ldots,f_n\in F^1$ ($y$ is not a vertex),
						\item
						$yy'\in \vartheta(s_E^{-1}(v))\,\Rightarrow\, y'=t_{PF}(y)$,
						\item
						$\forall\, i\in\{1,\ldots,n\}, f\in s_F^{-1}(s_F(f_i))\,\exists\,y'\in  \mathrm{FP}(F):f_1\ldots f_{i-1}fy'\in\vartheta(s_E^{-1}(v))$.
					\end{enumerate}
				\end{enumerate}
				\item 
				For any $v\in {\rm reg}_0(E)$, either the above condition holds or $\vartheta(s^{-1}_E(v))=\vartheta(v)$.
			\end{enumerate}
		\end{enumerate}
	\end{definition}
	
	Note that, even though $\mathsf{OG}$ is a subcategory of $\mathsf{PG}$, the regularity conditions in these two categories have different meanings. 
	Here is the list of categories used in \cite{HT24,HT24cov}. The class of objects is always the same, i.e.\ all graphs.
	\begin{enumerate}[label={$\bullet$}]
		\item 
		$\mathsf{POG}$ is the subcategory of $\mathsf{OG}$ consisting of proper graph homomorphisms;
		\item 
		$\mathsf{TBPOG}$ is the subcategory of $\mathsf{POG}$ consisting of 
		target-bijective proper graph homomorphisms;
		\item 
		$\mathsf{CRTBPOG}$ is the subcategory of $\mathsf{TBPOG}$ consisting of 
		regular target-bijective proper graph homomorphisms;
		\item 
		$\mathsf{IPG}$ is the subcategory of $\mathsf{PG}$ consisting of path homomorphisms that are injective on vertices;
		\item 
		$\mathsf{MIPG}$ is the subcategory of $\mathsf{IPG}$ consisting of path homomorphisms that are monotone and injective on vertices;
		\item 
		$\mathsf{RMIPG}$ is the subcategory of $\mathsf{MIPG}$ consisting of path homomorphisms that are regular monotone, and injective on vertices.
	\end{enumerate}
	\noindent
	The relationship between these subcategories is depicted by the diagram with only one vertical inclusion:
	\begin{center}
		\begin{tikzpicture}[yscale=2,xscale=-2.7]
			
			\node (a1) at (0,0) {$\mathsf{PG}$};
			\node (a2) at (1,0) {$\mathsf{IPG}$.};
			\node (a3) at (2,0) {$\mathsf{MIPG}$};
			\node (a4) at (3,0) {$\mathsf{RMIPG}$};
			\node (b1) at (0,1) {$\mathsf{OG}$};
			\node (b2) at (1,1) {$\mathsf{POG}$};
			\node (b3) at (2,1) {$\mathsf{TBPOG}$};
			\node (b4) at (3,1) {$\mathsf{CRTBPOG}$};
			
			\path[right hook-To]
			(a2) edge (a1)
			(a3) edge (a2)
			(a4) edge (a3)
			(b2) edge (b1)
			(b3) edge (b2)
			(b4) edge (b3)
			(b1) edge (a1);
			
		\end{tikzpicture}
		\vspace*{-3mm}\end{center}

\subsection{Algebras associated to graphs}

We now pass to  definitions of the algebras determined by graphs. We denote by \mbox{$k$-$\mathsf{Alg}$}
the category of associative algebras over the field $k$ and algebra homomorphisms, and
we denote by C*-$\mathsf{Alg}$ the category of C*-algebras and *-homomorphisms.

Given a non-empty graph $E$, let $kE$ be the vector space of finitely supported maps $ \mathrm{FP}(E)\to k$ with pointwise addition and multiplication. 
For $x\in  \mathrm{FP}(E)$, denote by $S_x$ the characteristic function of 
the singleton $\{x\}$, i.e.\ $S_x(x')=1$ if $x=x'$ and $S_x(x')=0$ otherwise. The set $\{S_x\mid x\in  \mathrm{FP}(E)\}$ is a vector space basis of $kE$.
If $E$ is the empty graph, we adopt the convention that $k E=\{0\}$.

\begin{definition}[\protect{\cite[Definition~II.1.2]{ass-06}}]
The vector space $kE$ with multiplication defined on basis elements by
\[
S_xS_{x'}:=\begin{cases}
S_{xx'} & \text{if }t_{PE}(x)=s_{PE}(x') \\ 0 & \text{otherwise}
\end{cases}
\]
is an associative algebra called the \emph{path algebra} of $E$.
\end{definition}

Note that, if $x=v$ is a vertex, an alternative notation $P_v$ is used for the generator $S_v$ is  to stress that it is an idempotent element of the algebra. 
Observe also that the path algebra can be defined 
as the universal associative $k$-algebra generated by $\{S_x\mid x\in E^0\cup E^1\}$ subject to the relations:
\begin{itemize}[leftmargin=3em]
	\item[(V)] $S_vS_w=\delta_{v,w}S_v$ for all $v,w\in E^0$,
	\item[(E1)] $S_{s_E(e)}S_e=S_e=S_eS_{t_E(e)}$ for all $e\in E^1$.
\end{itemize}

Next, we can refine the basic concept of a path algebra as follows:
\begin{definition}[\protect{\cite[Definition~1.5.1]{aasm17}}]\label{def:rela}
Let $E$ be a graph. The \emph{Cohn path algebra} $C_k(E)$ is the universal associative $k$-algebra generated by 
$\{S_x\mid x\in E^0\cup E^1\}\sqcup\{S_{e^*}\mid e\in E^1\}$ subject to the relations (V), (E1) together with
\begin{itemize}[leftmargin=3em]
	\item[(E2)] $S_{t_E(e)}S_{e^*}=S_{e^*}=S_{e^*}S_{s_E(e)}$ for all $e\in E^1$,
	\item[(CK1)]\label{ck1} $S_{e^*}S_{e'}=\delta_{e,e'}S_{t_E(e)}$ for all $e,e' \in E^1$.
\end{itemize}
Better still, we can now pass to a main object of interest, which is:
\end{definition}
\begin{definition}[\protect{\cite[Definition~1.2.3]{aasm17}}]
The \emph{Leavitt path algebra} $L_k(E)$ is the universal associative $k$-algebra generated by $\{S_x\mid x\in E^0\cup E^1\}\sqcup\{S_{e^*}\mid e\in E^1\}$ 
subject to the relations (V), (E1), (E2), (CK1) together with
\begin{itemize}[leftmargin=3em]
	\item[(CK2)]\label{ck2} $S_v=\sum_{e\in s^{-1}_E(v)}S_eS_{e^*}$ for all $v\in \reg(E)$.
\end{itemize}
\end{definition}
\begin{remark}
We can also define $C_k(E)$ and $L_k(E)$ using the notion of an extended graph.
The \emph{extended graph} $\bar{E}:=(\bar{E}^0, \bar{E}^1,s_{\bar{E}},t_{\bar{E}})$
of the graph $E$ is given as follows:
\begin{gather}
	\bar{E}^0:=E^0,\quad \bar{E}^1:=E^1\sqcup (E^1)^*,\quad (E^1)^*:=\{e^*~|~e\in E^1\},
	\nonumber\\
	\forall\; e\in E^1:\quad s_{\bar{E}}(e):=s_E(e),\quad t_{\bar{E}}(e):=t_E(e),
	\nonumber\\
	\forall\; e^*\in (E^1)^*:\quad s_{\bar{E}}(e^*):=t_E(e),\quad t_{\bar{E}}(e^*):=s_E(e).
\end{gather}
Then the Cohn path algebra $C_k(E)$ is the quotient of the path algebra $k\bar{E}$ of the extended graph $\bar{E}$ by the ideal generated by the set:
\[\label{en:Cohn}
	\left\{S_{e^*}S_{e'}-\delta_{e,e'}S_{t_E(e)} \;\big|\; e,e' \in E^1\right\}.
\]
Much in the same way, the Leavitt path algebra $L_k(E)$  is the path algebra $k\bar{E}$ of the extended graph $\bar{E}$ divided by the ideal generated by 
the union of the set \eqref{en:Cohn} with the set:
\[\label{en:Leavitt}
	\left\{\sum_{e\in s^{-1}_E(v)}S_eS_{e^*}-S_v \;\Big|\; v\in \reg(E)\right\}.
	\vspace*{-12pt}
\]
\end{remark}

To simplify the notation, we use the same symbol $S_e$ for an element in $k E$, in $k \bar{E}$, and its quotient class in $C_k(E)$ and in $L_k(E)$. 
If $\bar{x}=\bar{e}_1\cdots\bar{e}_n\in\mathrm{FP}(\bar{E})$, where \mbox{$\bar{e}_1,\,\ldots,\,\bar{e}_n\in\bar{E}^1$}, then
$S_{\bar{x}}:=S_{\bar{e}_1}\cdots S_{\bar{e}_n}$ and
$\bar{x}^*:=\bar{e}_n^*\cdots\bar{e}_1^*$, where $(e_i^*)^*:=e_i\in E^1$.
It will be always clear from the context which algebra $S_e$ is an element of.
Furthermore,
when the ground field is $\mathbb{C}$, the Cohn and the Leavitt path algebras of a graph $E$ are both *-algebras, with involution defined on generators by
 $S_v^*:=S_v$, $(S_e)^*:=S_{e^*}$ and $(S_{e^*})^*:=S_e$, which is consistent with the above general-field notation.
Since both the Cohn path *-algebra and the Leavitt path *-algebra are generated by partial isometries,
 whose norms cannot exceed one, the universal enveloping C*-algebra exist leading to the following definition:
\begin{definition}[\protect{\cite[Theorem~4.1]{fr99}, \cite[Definition~5.2.1]{aasm17}}]\label{df:gCs}
Let $E$ be a graph.
The universal enveloping C*-algebra  of the Cohn path *-algebra $C_{\mathbb{C}}(E)$ is called the \emph{Toeplitz graph C*-algebra} of $E$ and is 
denoted by $\mathcal{T}(E)$. The universal enveloping C*-algebra  of  the Leavitt path *-algebra $L_{\mathbb{C}}(E)$ is called the \emph{graph C*-algebra} of $E$ 
and is denoted by~$C^*(E)$.
\end{definition}

\begin{remark}\label{rem:gCs}
The equivalence of Definition~\ref{df:gCs} with the definition of $C^*(E)$ as the universal C*-algebra generated by a Cuntz--Krieger $E$-family is discussed, e.g.,
in \cite{t-m07}. Therein, it is also shown that
the canonical \mbox{*-homomorphism} $L_{\mathbb{C}}(E)\to C^*(E)$ is injective, so $L_{\mathbb{C}}(E)$ can be identified with a dense *-subalgebra of $C^*(E)$.
\end{remark}

We end the section by recalling the known functoriality of the above associations of algebras to graphs.
A morphism $\varphi:E\to F$ in $\mathsf{POG}$ induces an algebra homomorphism $\varphi^*:kF\to kE$
given by
\[
\varphi^*(S_y):=\sum_{x\in\varphi^{-1}(y)}S_x\,,\quad\forall\;y\in F^0\cup F^1 .
\]
A morphism $\vartheta:F\to E$ in $\mathsf{IPG}$ induces an algebra homomorphism $\vartheta_*:kF\to kE$
given by
\[
\vartheta_*(S_y):= S_{\vartheta(y)}\,,\quad\forall\;y\in F^0\cup F^1 .
\]
Note that, when we extend $\varphi^*$ or $\vartheta_*$ to $kF$ by multiplicativity, we find that the same formulas as above hold for any path $y\in  \mathrm{FP}(F)$.

The following associations of algebras to graphs and of algebra homomorphisms to graph homomorphisms  
define contravariant functors~\cite{HT24}:
\begin{enumerate}
\item 
$\mathsf{POG}\to k$-$\mathsf{Alg}$, $E\mapsto kE$, $\varphi\mapsto\varphi^*$;
\item 
$\mathsf{TBPOG}\to k$-$\mathsf{Alg}$, $E\mapsto C_k(E)$, $\varphi\mapsto\varphi^*$;
\item 
$\mathsf{CRTBPOG}\to k$-$\mathsf{Alg}$, $E\mapsto L_k(E)$, $\varphi\mapsto\varphi^*$.
\end{enumerate}
The following associations of algebras to graphs and of algebra homomorphisms to path homomorphisms of graphs
define covariant functors \cite{HT24cov}:
\begin{enumerate}
\item 
$\mathsf{IPG}\to k$-$\mathsf{Alg}$, $E\mapsto kE$, $\vartheta\mapsto\vartheta_*$;
\item 
$\mathsf{MIPG}\to k$-$\mathsf{Alg}$, $E\mapsto C_k(E)$, $\vartheta\mapsto\vartheta_*$;
\item 
$\mathsf{RMIPG}\to k$-$\mathsf{Alg}$, $E\mapsto L_k(E)$, $\vartheta\mapsto\vartheta_*$.
\end{enumerate}
Replacing the Cohn path algebra by the Toeplitz graph C*-algebra and replacing the Leavitt path algebra by the graph C*-algebra
yield both contravariant and covariant functors into the category C*-$\mathsf{Alg}$ of C*-algebras and *-homomorphisms.

\section{Relation morphisms}\label{sec:3}\noindent

\vspace*{-5mm}

\subsection{Sets}
Let $X$ and $Y$ be sets, and
let $R\subseteq X\times Y$. We view $R$ as a relation between $X$ and $Y$. 
Given sets $X,Y,Z$ and relations $R\subseteq X\times Y$ and $S\subseteq Y\times Z$, their composition $S\circ R\subseteq X\times Z$ is defined as
\[
S\circ R:=\big\{(x,z)\in X\times Z\mid \exists\;y\in Y:(x,y)\in R,(y,z)\in S\big\}.
\]
Also, the diagonal relation $\{(x,x)\;|\; x\in X\}\subseteq X\times X$ is the identity for the above composition. 
All this allows us to define the category $\mathsf{RESET}$ of sets and relations with the former being objects and the latter being morphisms.
The usual category $\mathsf{SET}$ of sets and maps can be viewed as a subcategory of $\mathsf{RESET}$ by writing a map $f:X\to Y$
as the relation
\[\label{relphi}
R_f: =\big\{(x,f(x))\mid x\in X\big\} .
\]
Much in the same way, the opposite category $\mathsf{SET}^{op}$ can be viewed as a subcategory of  $\mathsf{RESET}$ by writing
 a map  $g:Y\to X$ as the relation
\[\label{reltheta}
R^g: =\big\{(g(y),y)\mid y\in Y\big\}.
\]

Now,
in the category $\mathsf{SET}$, for any relation $R\subseteq X\times Y$, consider the pushout diagram
\begin{equation}\label{eq:pullpush}
\begin{tikzpicture}[scale=2.3,baseline=(current bounding box.center)]

\node (a1) at (135:1) {$X$};
\node (a2) at ($(45:1)+(135:1)$) {$R$};
\node (b1) at (0,0) {$X\coprod_RY.$};
\node (b2) at (45:1) {$Y$};

\path[-To,font=\footnotesize,inner sep=2pt,shorten <=2pt,shorten >=2pt]
(a1) edge[dashed] node[above=2pt] {$R_q\circ R^p$} node[below=2pt] {$R^j\circ R_i$} (b2)
(a2) edge node[above left] {$p$} (a1)
(b2) edge node[below right] {$j$} (b1)
(a1) edge node[below left] {$i$} (b1)
(a2) edge node[above right] {$q$} (b2);

\end{tikzpicture}
\end{equation}
given by  the canonical projections $p$ and~$q$.
 One can check that we always have the (pullback-type) factorization
$R=R_q\circ R^p$. However, the (pushout-type) factorization $R=R^j\circ R_i$ holds if and only if $R$ is transitively closed in the sense explained below.
(In general, we only have $R\subseteq R^j\circ R_i$.)
Now, the \emph{transitive closure} of the relation $R\subseteq X\times Y$ is the smallest equivalence relation in
\[
(X\sqcup Y)\times (X\sqcup Y)=(X\times X)\sqcup(X\times Y)\sqcup(Y\times X)\sqcup(Y\times Y) 
\]
containing $R$ intersected with~$X\times Y$. More explicitly, one can obtain the transitive closure of $R$ by completing all triples in $R$ of the form
 $\{(x,y), (x',y), (x',y')\}$ to quadruples $\{(x,y), (x',y), (x',y'), (x,y')\}$. Another way to look at the transitive closure is to say that it is the
 pullback of 
\[\label{pushset}
X\xrightarrow{\;i\;}X\coprod_RY\xleftarrow{\;j\;}Y.
\]
 A relation is called \emph{transitively closed} if it is equal to its transitive closure.
 Note that all relations of the form $R_f$ or $R^g$ are always transitively closed. All this leads to the following criterion
 for the existence of a pushout-type factorization:
 \begin{proposition}\label{criterion}
A relation $R\subseteq X\times Y$ can be written as the composition $R=R^b\circ R_a$, where $a\in \mathrm{Map}(X,Z)$ and $b\in \mathrm{Map}(Y,Z)$
for some set~$Z$, if and only if it is transitively closed.
 \end{proposition}
\begin{proof}
If $R$ is transitively closed, then we can take the pushout \eqref{pushset} to write $R=R^j\circ R_i$. Vice versa, if $R=R^b\circ R_a$ and
$(x,y),\,(x,y'),\,(x',y')\in R$, then
\[
b(y)=a(x)=b(y')=a(x').
\]
Hence, $(x',y)\in R$, which proves that $R$ is transitively closed.
\end{proof}
Let us end these generalities by 
defining the relation image of $x\in X$ as
\[
R(x):=\big\{y\in Y\mid (x,y)\in R\big\} ,
\]
and the relation preimage of $y\in Y$ as
\[
R^{-1}(y):=\big\{x\in X\mid (x,y)\in R\big\} .
\]
We extend these notions to subsets in the obvious way.

\subsection{Graphs}

Let us now introduce an elementary but fundamental concept of this paper:
\begin{definition}
Let $E$ and $F$ be graphs. A \emph{relation morphism}
$R:E\to F$ is a relation
\begin{equation*}
R\subseteq  \mathrm{FP}(E) \times  \mathrm{FP}(F)
\end{equation*}
such that
\begin{enumerate}
\item if $(x,y)\in R$ and $y\in F^0$, then $x\in E^0$ (i.e.\ $R^{-1}(F^0)\subseteq E^0$);
\item if $(x,y)\in R$, then $(s_{PE}(x),s_{PF}(y))\in R$ (source preserving);
\item if $(x,y)\in R$, then $(t_{PE}(x),t_{PF}(y))\in R$ (target preserving).
\end{enumerate}
\end{definition}

One can easily verify that this notion of morphism leads to a category. We denote by $\Id_E$ the identity morphism on $E$.

\begin{definition}
We call $\mathsf{RG}$ the category whose objects are all graphs and 
whose morphisms are relation morphisms.
\end{definition}

It is straightforward to prove the following:
\begin{proposition}
Let $\varphi: \mathrm{FP}(E)\to  \mathrm{FP}(F)$ be a morphism in the category $\mathsf{OG}$ and $\vartheta: \mathrm{FP}(F)\to  \mathrm{FP}(E)$ be a morphism in the category~$\mathsf{PG}$.
Then the assignments 
$$
\varphi\longmapsto R_\varphi\,,\quad \vartheta\longmapsto R^\vartheta\,,
$$
defined, respectively, by \eqref{relphi} and \eqref{reltheta},  yield, respectively, the embeddings of categories 
$\mathsf{OG}$ and $\mathsf{PG}^{\mathrm{op}}$ in $\mathsf{RG}$ that are identity on objects. 
\end{proposition}

Inspired by the diagram \eqref{eq:pullpush}, we now discuss factorizations of relation morphisms.
\begin{definition}\label{eq:pull}
Given a relation morphism $R:E\to F$, we define the \emph{relation graph} $G_R$ as follows:
\begin{gather*}
G_R^0:=(E^0\times F^0)\cap R , \quad
G_R^1:=\big( \mathrm{FP}(E)\times F^1\big)\cap R, \nonumber\\
s_{G_R}:=\big(s_{PE},s_{F}\big) , \quad
t_{G_R}:=\big(t_{PE},t_{F}\big) .
\end{gather*}
\end{definition}
To obtain a desired factorization, we need two additional properties.
\begin{definition}
A relation morphism $R:E\to F$ is called \emph{multiplicative} whenever 
$$
(x,y),(x',y')\in R,\; s_{PE}(x')=t_{PE}(x),\; s_{PF}(y')=t_{PF}(y)\quad
\Rightarrow \quad (xx',yy')\in R.
$$
\end{definition}
\begin{definition}
A relation morphism $R:E\to F$ is called \emph{decomposable}
whenever 
$$
y,y'\in  \mathrm{FP}(F),\, s_{PF}(y')=t_{PF}(y),\, (\ol{x},yy')\in R\;\Rightarrow\;\exists\;x\in R^{-1}(y),\,x'\in R^{-1}(y')\colon xx'=\ol{x}.
$$
 \end{definition}
 
We are now ready to claim:
\begin{proposition}\label{prop:factor}
Let $R:E\to F$ be a relation morphism that is multiplicative and decomposable.
Then there exists the following (pullback-type) factorization
\begin{equation*}
\begin{tikzpicture}[scale=2.5,baseline=(current bounding box.center)]

\node (a1) at (135:1) {$E$};
\node (a2) at ($(45:1)+(135:1)$) {$G_R$};
\node (b2) at (45:1) {$F\,.$};

\path[-To,font=\footnotesize,inner sep=2pt,shorten <=2pt,shorten >=2pt]
(a1) edge node[above=2pt] {$R=R_\varphi\circ R^\vartheta$} (b2)
(a2) edge node[above left] {$\vartheta$} (a1)
(a2) edge node[above right] {$\varphi$} (b2);

\end{tikzpicture}
\end{equation*}
Here $G_R$ is the relation graph defined above, $\vartheta:G_R\to E$ is the morphism in $\mathsf{PG}$ induced by the projection $R\to  \mathrm{FP}(E)$ on the first 
component, and $\varphi:G_R\to F$ is the morphism in $\mathsf{OG}$ induced by the projection $R\to  \mathrm{FP}(F)$ on the second component.
\end{proposition}
\begin{proof}{}
[$R\subseteq R_\varphi\circ R^\vartheta$]
Let $(x,y)\in R$. If $y\in F^0\cup F^1$, then $(x,y)\in G_R^0\cup G_R^1$, $\vartheta(x,y)=x$ and $\varphi(x,y)=y$, so $(x,y)\in R_\varphi\circ R^\vartheta$. 
Assume now that $n:=|y|\geq 2$, and write $y=f_1\cdots f_n$, where all $f_i$ are edges. 
By the decomposability of $R$, there are $x_1,\ldots,x_n\in  \mathrm{FP}(E)$ such that $x=x_1\cdots x_n$ and 
$(x_i,f_i)\in R$ for all $i=1,\ldots,n$. In this case, $z:=(x_1,f_1)\cdots (x_n,f_n)\in  \mathrm{FP}(G_R)$ and $z$ is such that $\vartheta(z)=x$ and $\varphi(z)=y$. 
Hence $(x,y)\in R_\varphi\circ R^\vartheta$.

[$R\supseteq R_\varphi\circ R^\vartheta$]
Assume now that $(x,y)\in R_\varphi\circ R^\vartheta$, and let $z\in  \mathrm{FP}(G_R)$ be such that $\vartheta(z)=x$ and $\varphi(z)=y$. Since $\varphi$ preserves lengths, 
if $z\in G_R^0\cup G_R^1$, then $z=(x,y)\in R$ by construction. In the case $n:=|z|\geq 2$, we can write $z=(x_1,f_1)\cdots (x_n,f_n)$, where $(x_i,f_i)\in G_R^1$ 
for all $i=1,\ldots,n$. Furthermore, $x=x_1\cdots x_n$ and $y=f_1\cdots f_n$. Now it follows from the multiplicativity of $R$ that $(x,y)\in R$.
\end{proof}

The construction in the above proposition satisfies the following, pullback-like, universal property.

\begin{proposition}\label{prop:universal1}
Let $R:E\to F$ be a relation morphism  that is multiplicative and decomposable.
For any $\varphi'\in\mathsf{OG}(G',F)$ and $\vartheta'\in\mathsf{PG}(G',E)$ such that $R_{\varphi'}\circ R^{\vartheta'}=R$ there exists a unique 
$\pi\in\mathsf{OG}(G',G_R)$ rendering the  diagram 
\begin{center}
\begin{tikzpicture}[xscale=3,yscale=1.5,baseline=(current bounding box.center)]

\node (a) at (0,0) {$E$};
\node (b) at (1,1) {$G_R$};
\node (c) at (1,2) {$G'$};
\node (d) at (2,0) {$F\,.$};

\path[-To,font=\footnotesize,inner sep=2pt,shorten <=2pt,shorten >=2pt]
(a) edge[dashed] node[above=2pt] {$R$} (d)
(b) edge node[below right,pos=0.4] {$\vartheta$} (a)
(b) edge node[below left,pos=0.4] {$\varphi$} (d)
(c) edge node[above left] {$\vartheta'$} (a)
(c) edge node[above right] {$\varphi'$} (d)
(c) edge node[right] {$\pi$} (b);

\end{tikzpicture}
\end{center}
commutative.
Furthermore, the unique map $\pi$ is surjective.
\end{proposition}

\begin{proof}
The equality $R_{\varphi'}\circ R^{\vartheta'}=R$ means that
\[
R=\big\{\big(\vartheta'(z),\varphi'(z)\big)\mid z\in G'\big\} .
\]
For $z\in  \mathrm{FP}(G')$, we define $\pi(z):=\big(\vartheta'(z),\varphi'(z)\big)$. It is clear from the definition of $G_R$ that the thus defined $\pi$ is a morphism in 
$\mathsf{OG}(G',G_R)$. The
surjectivity and uniqueness of~$\pi$, and the commutativity of the diagram are obvious.
\end{proof}

The map $\pi$ in previous proposition is not always injective, as shown in the next example.

\begin{example}
Consider the commutative diagram:
\begin{center}
\begin{tikzpicture}[xscale=2.7,yscale=2.3]

\node[draw,rectangle,color=gray!60] (a1) at (0,0) {\includegraphics[page=30,valign=c]{images.pdf}};
\node[draw,rectangle,color=gray!60] (a2) at (2,0) {\includegraphics[page=31,valign=c]{images.pdf}};
\node[draw,rectangle,color=gray!60] (b1) at (1,1) {\includegraphics[page=33,valign=c]{images.pdf}};
\node[draw,rectangle,color=gray!60] (b2) at (1,-1) {\includegraphics[page=32,valign=c]{images.pdf}};

\path[-To,font=\footnotesize,inner sep=1pt,shorten <=2pt,shorten >=2pt]
(b2) edge node[below left] {$\vartheta$} (a1)
(b1) edge node[above right] {$\varphi'$} (a2)
(b1) edge node[right] {$\pi$} (b2)
(b1) edge node[above left] {$\vartheta'$} (a1)
(b2) edge node[below right] {$\varphi$} (a2);

\draw
(a1.north) node[above left] {$E$}
(b2.west) node[left] {$G_R$}
(a2.north) node[above] {$F$}
(b1.west) node[left] {$G'$};

\end{tikzpicture}.
\end{center}
Here $ \vartheta'$ and $\varphi'$ are  identity on vertices, $\vartheta'(g)=\vartheta'(g')=e$, $\varphi'(g)=\varphi'(g')=f$.
The maps $\vartheta$, $\varphi$ and $\pi$ are  given by the general construction. In particular, $\pi(g)=\pi(g')=(e,f)$.
\end{example}

In contradistinction with the above construction, we cannot mimic the pushout construction for graphs. The prime obstruction
is already given by Proposition~\ref{criterion}. A more subtle  stumbling block is as follows:
\begin{proposition}\label{prop:nogo}
Let $R:E\to F$ be a relation morphism. Suppose that there exists
$f\in F^1$ such that $R^{-1}(f)$ contains two paths $x,x'$ of different length.
Then, there is no (pushout-type) factorization
\begin{center}
\begin{tikzpicture}[scale=2.3,baseline=(current bounding box.center),inner sep=2pt]

\node (a1) at (135:1) {$E$};
\node (b1) at (0,0) {$H$};
\node (b2) at (45:1) {$F$};

\path[-To,font=\footnotesize,inner sep=2pt,shorten <=2pt,shorten >=2pt]
(a1) edge node[above=2pt] {$R^{\vartheta'}\circ R_{\varphi'}$} (b2)
(b2) edge node[below right] {$\vartheta'$} (b1)
(a1) edge node[below left] {$\varphi'$} (b1);

\end{tikzpicture}
\end{center}
for any graph $H$ and any morphisms $\varphi':E\to H$ in $\mathsf{OG}$ and $\vartheta':F\to H$ in $\mathsf{PG}$.
\end{proposition}
\begin{proof}
Suppose such a factorization exists. Then $\varphi'(x)=\vartheta'(f)=\varphi'(x')$, but $\varphi'$ is length preserving, which gives a contradiction.
\end{proof}

An explicit counterexample to the existence of the above factorization is the following.
\begin{example}\label{ex:counter}
Consider the relation morphism
\begin{center}
\begin{tikzpicture}[scale=5]

\node[draw,rectangle,color=gray!60] (a) at (0,0) {\includegraphics[page=25,valign=c]{images.pdf}};
\node[draw,rectangle,color=gray!60] (b) at (1,0) {\includegraphics[page=24,valign=c]{images.pdf}};

\path[-To,font=\footnotesize,inner sep=1pt,shorten <=2pt,shorten >=2pt]
(a) edge node[above=2pt] {$R$} (b);
\end{tikzpicture},
\end{center}
where $R:=\big\{ (1,1),(2,2),(3,1),(5,2),(e,f),(e'e'',f) \big\}$.
Then $R^{-1}(f)$ contains paths of length $1$ and $2$, thus satisfying the assumptions of Proposition~\ref{prop:nogo} 
preventing the existence of a pushout-type factorization.
On the other hand, the pullback-type factorization of Proposition~\ref{prop:factor} in the present example becomes
\begin{center}
\begin{tikzpicture}[scale=5]

\node[draw,rectangle,color=gray!60] (a1) at (0,0) {\includegraphics[page=25,valign=c]{images.pdf}};
\node[draw,rectangle,color=gray!60] (a2) at (1,0) {\includegraphics[page=29,valign=c]{images.pdf}};
\node[draw,rectangle,color=gray!60] (b2) at (2,0) {\includegraphics[page=24,valign=c]{images.pdf}};

\path[-To,font=\footnotesize,shorten <=2pt,shorten >=2pt]
(a1) edge node[above] {$R^\vartheta$} (a2)
(a2) edge node[above] {$R_\varphi$} (b2);

\end{tikzpicture}.
\end{center}
\vspace*{-20pt}
\end{example}

We end this section with a technical lemma allowing us to define a relation morphism $R:E\to F$ specifying it on vertices and edges of $F$ alone.
\begin{lemma}\label{lem:generation}
Let $E,F$ be graphs, and $S^0\subseteq E^0\times F^0$, $S^1\subseteq  \mathrm{FP}(E)\times F^1$ be subsets such that
$(s_{PE}(x),s_{F}(f)),(t_{PE}(x),t_F(f))\in S^0$ for all $(x,f)\in S^1$.
Then
there exists a unique multiplicative and decomposable
relation morphism $R:E\to F$ such that $G_R^0=S^0$ and $G_R^1=S^1$.
Every multiplicative and decomposable
relation morphism arises in this way.
\end{lemma}
\begin{proof}
Given $S^0$ and $S^1$, define
\begin{align}
& R:=\{(x,y)\in  \mathrm{FP}(E) \times  \mathrm{FP}(F)\mid (x,y)\in S^0 \notag \\[2pt]
& \hspace*{2cm} \text{ or } \exists\; (x_1,f_1),\ldots,(x_n,f_n)\in S^1:x=x_1\cdots x_n,y=f_1\cdots f_n\}.
\end{align}
By construction $G_R^0=S^0$, $G_R^1=S^1$ and $R$ is decomposable. It is also easy to check that $R$ is multiplicative, and $R$ preserves sources 
and targets. The uniqueness of such $R$ follows from decomposability and multiplicativity.
On the other hand, any multiplicative and decomposable relation $R$ can be recovered from $G_R^0$ and $G_R^1$ as above.
\end{proof}

Note that, if $R:E\to F$ is a multiplicative and decomposable relation morphism, then $R$ is generated by $G_R^0$ and $G_R^1$ 
by the componentwise concatenation, so $R$ can be viewed as the set of finite paths in~$G_R$.

\section{Path algebras}\label{sec:4}
\noindent
Much as it was done in \cite{HT24} and \cite{HT24cov}, now we want to determine a subcategory of $\mathsf{RG}$ admitting a functor to $k$-$\mathsf{Alg}$ by 
assigning path algebras to graphs
and algebra homomorphisms to relation morphisms. Unfortunately, it turns out that there is no such a subcategory of $\mathsf{RG}$ that contains both 
$\mathsf{IPG}^{\mathrm{op}}$ and $\mathsf{POG}$. Therefore, we need to settle for a subcategory that contains $\mathsf{IPG}^{\mathrm{op}}$ but does not 
contain $\mathsf{POG}$.

\subsection{The subcategory $\mathsf{PRG}$ and the path-algebra functor}

The next two steps towards determining admissible relation morphisms inducing *-homomorphisms between C*-algebras are:
\begin{definition}
Given a relation morphism $R:E\to F$, we say that it is \emph{proper} if $|R^{-1}(y)|<\infty$ for all $y\in  \mathrm{FP}(F)$,
and we say that it is \emph{vertex disjoint} if 
$$
(v,\,v'\in F^0\text{ and } v\neq v') \;\Rightarrow\; R^{-1}(v)\cap R^{-1}(v')=\emptyset.
$$
\end{definition}
Now we are ready to provide the homomorphism formula used throughout the paper. It is worth emphasizing that this formula 
generalizes both the formula  for the contravariant  induction and the  covariant induction given in \cite{HT24} and \cite{HT24cov}, respectively.
\begin{lemma}\label{lem:RP}
Let $R:E\to F$ be a relation morphism that is multiplicative, 
decomposable, proper and vertex disjoint.
Then the formula
\begin{equation}\label{eq:Rstar}
R_P^*(S_y):= \sum_{x\in R^{-1}(y)}S_x,
\end{equation}
where $y\in F^0\cup F^1$, defines an algebra homomorphism $R_P^*:kF\to kE$.
\end{lemma}
\begin{proof}
For $v\in F^0$, the element 
$R_P^*(S_v)$ is a sum of orthogonal idempotents, so it is an idempotent itself. For $v,v'\in F^0$ such that $v\neq v'$, since $R$ is vertex 
disjoint, we infer that $R_P^*(S_v)R_P^*(S_{v'})=0$. Also, since $R$ preserves both targets and sources, we obtain $R_P^*(S_{s_F(f)})R_P^*(S_f)=R_P^*(S_f)$ and 
$R_P^*(S_f)R_P^*(S_{t_F(f)})=R_P^*(S_f)$ for all $f\in F^1$. By the universal property of $kF$, \eqref{eq:Rstar} defines a homomorphism $R_P^*:kF\to kE$.
\end{proof}

As seen in the above proof, $R_P^*(S_y)$ is defined on paths $y$ of length at least 2 by using the universal property of path algebras.
In general, however, the formula \eqref{eq:Rstar} may not hold for paths of length at least~2. A counterexample is below. 
\begin{example}
Consider the following relation morphism:
\begin{center}
\begin{tikzpicture}[scale=6]

\node[draw,rectangle,color=gray!60] (a) at (0,0) {\includegraphics[page=21,valign=c]{images.pdf}};
\node[draw,rectangle,color=gray!60] (c) at (1,0) {\includegraphics[page=23,valign=c]{images.pdf}};

\path[-To,font=\footnotesize,inner sep=1pt,shorten <=2pt,shorten >=2pt]
(a) edge node[above=2pt] {$R$} (c);

\end{tikzpicture}
\end{center}
Here $R$ is the relation morphism generated by
\[
G_R^0:=\{(1,1),(2,2),(3,3)\} ,
\qquad 
G_R^1:=\big\{(e,f),(ee'',f),(e',f'),(e''e',f')\big\} .
\]
A simple computation gives
\begin{align}
R_P^*(S_{ff'})=R_P^*(S_f)R_P^*(S_{f'})
&=S_{ee'}+2S_{ee'' e'}+S_{e(e'')^2e'} \notag\\
&\neq \sum_{x\in R^{-1}(ff')}S_x=
S_{ee'}+S_{ee'' e'}+S_{e(e'')^2e'} .
\end{align}
\vspace*{-25pt}
\end{example}

The next example shows that there is no  subcategory of $\mathsf{RG}$ containing both $\mathsf{IPG}^{\mathrm{op}}$ and $\mathsf{POG}$ such that the  
association $R\mapsto R_P^*$ of Lemma~\ref{lem:RP} is functorial.

\pagebreak

\begin{example}\label{ex:45}
Consider the following relation morphisms:
	\begin{center}
		\begin{tikzpicture}[scale=4.5]
			
			\node[draw,rectangle,color=gray!60] (a) at (0,0) {\includegraphics[page=21,valign=c]{images.pdf}};
			\node[draw,rectangle,color=gray!60] (b) at (1,0) {\includegraphics[page=22,valign=c]{images.pdf}};
			\node[draw,rectangle,color=gray!60] (c) at (2,0) {\includegraphics[page=26,valign=c]{images.pdf}};
			\node[draw,rectangle,color=gray!60] (d) at (2.8,0) {\includegraphics[page=27,valign=c]{images.pdf}};
			
			\path[-To,font=\footnotesize,inner sep=1pt,shorten <=2pt,shorten >=2pt]
			(a) edge node[above=2pt] {$R^\vartheta$} (b)
			(b) edge node[above=2pt] {$R_\varphi$} (c)
			(c) edge node[above=2pt] {$R^{\vartheta'}$} (d);
		\end{tikzpicture}
	\end{center}
	Here $\vartheta$ and $\vartheta'$ are morphisms in 
	$\mathsf{IPG}$, $\varphi$ is morphism in $\mathsf{POG}$,
	and they are defined by the formulas:
	\begin{gather}
		\vartheta(f)=e , \quad \vartheta(\widetilde{f})=ee'' , \quad
		\vartheta(f')=e' , \quad \vartheta(\widetilde{f}')=e''e' , \notag\\
		\varphi(f)=\varphi(\widetilde{f})=g , \qquad
		\varphi(f')=\varphi(\widetilde{f}')=g' , \notag\\
		\vartheta'(h)=gg' .
	\end{gather}
	We now show that
	\[
	(R^{\vartheta'}\circ R_\varphi\circ R^\vartheta)_P^*\neq (R^\vartheta)_P^*\circ (R_\varphi)_P^*\circ (R^{\vartheta'})_P^*.
	\]
	Indeed,
	\[
	(R^{\vartheta'}\circ R_\varphi\circ R^\vartheta)_P^*(h)=S_{ee'}+S_{ee'' e'}+S_{e(e'')^2 e'}
	\]
	is different from
	\[
	(R^\vartheta)^*\circ (R_\varphi)_P^*\circ (R^{\vartheta'})_P^*(h)=S_{ee'}+2S_{ee'' e'}+S_{e(e'')^2 e'} 
	\]
	because $S_{ee'' e'}$ is a basis element, whence not zero.
\end{example}

We now are going to construct step by step a suitable subcategory of $\mathsf{RG}$.
\begin{lemma}\label{comp1}
Let $R:E\to F$ and $S:F\to G$ be relation morphisms. Then the following statements hold:
\begin{enumerate}
\item 
If $R$ and $S$ are decomposable, then so is $S\circ R$.
\item 
If $R$ and $S$ are proper, then so is $S\circ R$.
\item 
If $R$ and $S$ are vertex disjoint, then so is $S\circ R$.
\item 
If $R$ and $S$ are multiplicative and $R$ is vertex disjoint, then $S\circ R$ is multiplicative.
\end{enumerate}
\end{lemma}
\begin{proof}
(1) Let $z,z'\in  \mathrm{FP}(G)$ and $\ol{x}\in  \mathrm{FP}(E)$ be such that $s_{PG}(z')=t_{PG}(z)$ and $(\ol{x},zz')\in S\circ R$. By the definition of composition, 
there exists $\ol{y}\in  \mathrm{FP}(F)$ such that $(\ol{x},\ol{y})\in R$ and $(\ol{y},zz')\in S$. Using the decomposability of~$S$, 
we find $y,y'\in  \mathrm{FP}(F)$ such 
that $yy'=\ol{y}$ and $(y,z),(y',z')\in S$. In particular, $s_{PF}(y')=t_{PF}(y)$. Now, using the decomposability of~$R$, we find $x,x'\in  \mathrm{FP}(E)$ such that 
$xx'=\ol{x}$ and $(x,y),(x',y')\in R$. Hence, $(x,z),(x',z')\in S\circ R$, so $S\circ R$ is decomposable.

(2) and (3) are immediate.

(4) Let $(x,z),(x',z')\in S\circ R$ be such that $s_{PE}(x')=t_{PE}(x)$ and $s_{PG}(z')=t_{PG}(z)$. Hence, there are 
$y,y'\in  \mathrm{FP}(F)$ such that $(x,y),(x',y')\in R$ and $(y,z),(y',z')\in S$. Since $R$ preserves sources and targets, we conclude  that
\[
(t_{PE}(x),t_{PF}(y)),(s_{PE}(x'),s_{PF}(y'))\in R. 
\]
Next, from the facts that $s_{PE}(x')=t_{PE}(x)$ and that $R$ is vertex disjoint, we infer  that $s_{PF}(y')=t_{PF}(y)$. Then,
 the multiplicativity of $R$ and $S$ implies that $(xx',yy')\in R$ and $(yy',zz')\in S$. Therefore, $(xx',zz')\in S\circ R$, so $S\circ R$ is multiplicative.
\end{proof}

As we know from Example~\ref{ex:45}, the above conditions are not enough to define the desired path-algebra functor. It turns out that we need one more 
pivotal condition on relation morphisms:
\begin{definition}
A relation morphism $R:E\to F$ is called \emph{target injective}
if the restriction $t_{PE}:R^{-1}(f)\to E^0$
is injective for all $f\in F^1$.
\end{definition}

We are now ready for the second composability lemma:
\begin{lemma}\label{lem:target injective}
Let $R:E\to F$ and $S:F\to G$ be vertex-disjoint and target-injective relation morphisms. Then, if $R$ is also decomposable, $S\circ R$ is target injective.
\end{lemma}
\begin{proof}
Let $(x,g),(x',g)\in (S\circ R)$ be such that $g\in G^1$ and $t_{PE}(x)=t_{PE}(x')$. By the 
definition of composition, there exist $y,y'\in  \mathrm{FP}(F)$ such that 
$(x,y),(x',y')\in R$ and $(y,g),(y',g)\in S$. Since $R$ preserves targets and is vertex disjoint, we see that $t_{PF}(y)=t_{PF}(y')$. Now, 
using the target-injectivity of~$S$, we obtain  $y=y'$. Next, we need to
consider three cases depending on the length of~$y$.

If $y\in F^0$, then $x,x'\in E^0$ because 
$R^{-1}(F^0)\subseteq E^0$. In this case 
\begin{equation}
x=t_{PE}(x)=t_{PE}(x')=x'.
\end{equation}
If $y\in F^1$, then $x=x'$ since $R$ is target injective.
Finally, assume that $y=f_1\cdots f_n$ for some $n\geq 2$ and $f_1,\ldots,f_n\in F^1$. By the decomposability of $R$, there exist 
\[
x_1,\ldots,x_n,x'_1,\ldots,x'_n\in  \mathrm{FP}(E)
\]
 such that $x=x_1\ldots x_n$,  $x'=x_1'\ldots x_n'$ and $(x_i,f_i),(x'_i,f_i)\in R$ for all $i=1,\ldots,n$. Note that 
\[
t_{PE}(x_n)=t_{PE}(x)=t_{PE}(x')=t_{PE}(x'_n).
\]
Furthermore, since $R$ is target injective, we know  that $x_n=x_n'$. Now, for $i=1,\ldots,n-1$, we continue as follows.
From $x_{i+1}=x'_{i+1}$, we conclude that 
\[
t_{PE}(x_i)=s_{PE}(x_{i+1})=s_{PE}(x'_{i+1})=t_{PE}(x'_i)\,.
\]
Hence, as $R$ is target injective, we infer that $x_i=x'_i$, so
$x=x'$.
\end{proof}

Combining together the composition lemmas Lemma~\ref{comp1} and Lemma~\ref{lem:target injective}, we arrive at:
\begin{theorem}
	Relation morphisms that are multiplicative, decomposable, proper, vertex disjoint, and target injective define a subcategory of $\mathsf{RG}$ whose objects 
	are all directed graphs. (We denote this subcategory by $\mathsf{PRG}$.)
\end{theorem}

To construct a path-algebra functor from $\mathsf{PRG}$ to $k$-$\mathsf{Alg}$, first we need:
\begin{lemma}\label{lem:decomposition of composition}
Let $R:E\to F$ and $S:F\to G$ be relation morphisms such that
both are vertex disjoint and target injective, and $R$ is decomposable. Then
$(x,g)\in G_{(S\circ R)}^1$ implies that there is a unique $y\in S^{-1}(g)$ such that $(t_{PE}(x),t_{PF}(y))\in R$.
Furthermore, if $y=f_1\cdots f_n$ for some $n\geq 1$ and $f_1\ldots,f_n\in F^1$, then there are unique $x_1,\ldots,x_n\in  \mathrm{FP}(E)$ such that 
$x=x_1\ldots x_n$ and $(x_i,f_i)\in R$ for all $i=1,\ldots,n$.
\end{lemma}
\begin{proof}
The existence of $y$ comes from the definition of composition and the fact that $R$ preserves targets. 
The uniqueness is due to $R$ being vertex disjoint and $S$ being target injective.
The proof of second part of the statement follows the same idea as in the proof of Lemma~\ref{lem:target injective}, with $x=x_1\cdots x_n=x_1'\cdots x'_n$ 
decompositions of $x$ such that $(x_i,f_i),(x'_i,f_i)\in R$ for all $i=1,\ldots,n$.
\end{proof}

\begin{lemma}\label{decomposition}
	Let $R\in \mathsf{PRG}(E,F)$ and $S\in \mathsf{PRG}(F,G)$. Then the following two statements hold.
	\begin{enumerate}
		\item For every $w\in G^0$,
		\begin{equation*}
		\{u\in E^0\mid(u,w)\in (S\circ R)\}=\bigsqcup_{v\in S^{-1}(w)} \{u\in E^0\mid(u,v)\in R\}.
		\end{equation*}
		\item\label{decompostion.edges} For every $g\in G^1$,
		\begin{align*}
			\{x&\in  \mathrm{FP}(E)\mid(x,g)\in S\circ R\}=\bigsqcup_{v\in S^{-1}(g)\cap F^0}\{u\in E^0\mid(u,v)\in R\}\;\sqcup\\
			&\bigsqcup_{k=1}^n\ \bigsqcup_{y\in S^{-1}(g),\,|y|=k}\{x\in  \mathrm{FP}(E)\mid\exists!\, (x_1,f_1),\ldots, (x_k,f_k)\in R:x
			=x_1\cdots x_k,y=f_1\cdots f_k\},
		\end{align*}
		where $n=\max\{|y|\mid(y,g)\in S\}$.
	\end{enumerate}
\end{lemma}

\begin{proof}
	The first equality follows from the definition of composition and the fact that $R$ and $S$ are vertex disjoint. 
	The second equality follows from the definition of composition and Lemma~\ref{lem:decomposition of composition}.
\end{proof}

We are now ready to claim the main result of this section:
\begin{theorem}\label{thm:functor.PRG}
The association of path algebras to graphs and algebra homomorphisms $R_P^*$ to relation morphisms $R$ from the category $\mathsf{PRG}$
defines a contravariant functor from $\mathsf{PRG}$ to $k$-$\mathsf{Alg}$.
\end{theorem}

\begin{proof}
It is clear that $(\Id_E)_P^*=\Id_{kE}$ for any graph~$E$. Next,
let $R\in \mathsf{PRG}(E,F)$ and $S\in \mathsf{PRG}(F,G)$. Given $w\in G^0$, using Lemma~\ref{decomposition}, we compute:
\begin{align}
	(S\circ R)_P^*(S_w)&=\sum_{u\in (S\circ R)^{-1}(w)} S_u 
	=\sum_{v\in S^{-1}(w)}\ \sum_{u\in R^{-1}(v)} S_u\nonumber\\
	&=\sum_{v\in S^{-1}(w)} R_P^*(S_v) 
	=R_P^*\Big(\sum_{v\in S^{-1}(w)} S_v\Big) \nonumber\\
	&=R_P^*(S_P^*(S_w))
	=R_P^*\circ S_P^*(S_w).
\end{align}
Given $g\in G^1$, let $n=\max\{|y|\mid(y,g)\in S\}$. Then
\begin{align}
	S_P^*(S_g)&=\sum_{y\in R^{-1}(g)} S_y 
	=\sum_{v\in S^{-1}(g)\cap F^0}S_v\;+\;\sum_{k=1}^n\ \sum_{y\in S^{-1}(g),\,|y|=k}S_y.
\end{align}
Now, applying $R^*$ and using that it is an algebra homomorphism, we obtain
\begin{equation}\label{preserves.composition}
\hspace*{-5pt}
R_P^*(S_P^*(S_g))=\hspace{-5pt}\sum_{v\in S^{-1}(g)\cap F^0}R_P^*(S_v)\;+
\;\sum_{k=1}^n\sum_{\substack{f_1,\ldots,f_k\in F^1,\\f_1\cdots f_k\in S^{-1}(g)}}R_P^*(S_{f_1})\cdots R_P^*(S_{f_k}).
\end{equation}
Next, when expanding every term $R_P^*(S_{f_1})\cdots R_P^*(S_{f_k})$ in the above sum, 
we obtain a sum of terms of the form $S_{x_1}\cdots S_{x_k}$, where $(x_i,f_i)\in R$. Such terms differ
from zero only  when $x_1\ldots x_k\in  \mathrm{FP}(E)$, in which case $(x_1\ldots x_k,g)\in S\circ R$. Combining this observation together with 
Lemma~\ref{decomposition}\ref{decompostion.edges} and Equation~\eqref{preserves.composition}, we conclude  that
\[
(S\circ R)_P^*(S_g)=\sum_{x\in (S\circ R)^{-1}(g)}S_x=R_P^*(S_P^*(S_g)).
\]
Hence, $(S\circ R)_P^*=R_P^*\circ S_P^*$.
\end{proof}

\subsection{Subcategories and factorizations}

\begin{proposition}
The embedding of $\mathsf{PG}^{\mathrm{op}}$ into $\mathsf{RG}$ induces an embedding of $\mathsf{IPG}^{\mathrm{op}}$ into $\mathsf{PRG}$.
The embedding of $\mathsf{OP}$ into $\mathsf{RG}$ induces an embedding into $\mathsf{PRG}$ of 
the subcategory of $\mathsf{POG}$ obtained by restricting morphisms in $\mathsf{POG}$ to 
target-injective morphisms.
\end{proposition}

\begin{proof}
Let $\varphi\in\mathsf{POG}(E,F)$. Since $\varphi$ is length preserving and multiplicative, we see that $R_\varphi$ is decomposable, 
multiplicative and vertex disjoint. Properness and target-injectivity of $\varphi$ and $R_{\varphi}$ are easily seen to be equivalent.

Now, let $\vartheta\in\mathsf{IPG}^{\mathrm{op}}(E,F)$. Multiplicativity and decomposability are immediate. Furthermore, the relation $R^\vartheta$ is injective, 
so it is, in particular, target injective. For properness, observe that $|R^{-1}(y)|=1$ for all $y\in  \mathrm{FP}(F)$. 
Finally, we see that $R^\vartheta$ being vertex disjoint is the 
same as $\vartheta$ being injective on vertices.
\end{proof}

\begin{proposition}\label{prop:factunique}
Every $R\in\mathsf{PRG}(E,F)$ admits a factorization
$R=R_\varphi\circ R^\vartheta$, where $\vartheta\in\mathsf{IPG}(G,E)$ and $\varphi\in\mathsf{POG}(G,F)$ with
$\varphi$ target injective. This factorization is unique (up to a canonical isomorphism), and given explicitly by the construction in
Proposition~\ref{prop:factor}.
\end{proposition}
\begin{proof}
Consider  the factorization $R=R_\varphi\circ R^\vartheta$ given by Proposition~\ref{prop:factor}. To see that
$\vartheta\in\mathsf{IPG}(G,E)$, it remains to prove that 
$\vartheta$ is injective on vertices. Let $(u,v),(u',v')\in G_R^0$ be such that $\vartheta(u,v)=\vartheta(u',v')$, that is $u=u'$. Since $R$ is vertex disjoint, we infer 
that $v=v'$, whence $(u,v)=(u',v')$.

The fact that the properness of $R$ implies the properness of $\varphi$ is immediate. To see the target injectivity, take $f\in F^1$, and
note that $\varphi^{-1}(t_{F}(f))=\{(u,t_{F}(f))\in R\mid u\in E^0\}$ and $\varphi^{-1}(f)=\{(x,f)\in R\mid x\in  \mathrm{FP}(E)\}$. The map
\[
\varphi^{-1}(f)\ni(x,f)\longmapsto (t_{PE}(x),t_{F}(f))\in \varphi^{-1}(t_F(f))
\]
is injective because $R$ is target injective.

To show the uniqueness, take morphisms $\varphi'\in\mathsf{POG}(G',F)$ 
and $\vartheta'\in\mathsf{IPG}(G',E)$ such that $R_{\varphi'}\circ R^{\vartheta'}=R$, with $\varphi'$ target injective.
The canonical morphism $\pi\in\mathsf{OG}(G',G_R)$ of Proposition~\ref{prop:universal1} is surjective. For the injectivity,
as $\pi$ is a graph homomorphism, it is enough to prove the injectivity on vertices and edges. On vertices, it follows from the injectivity on vertices of $\vartheta'$.
Concerning edges, let $g_1,g_2\in G'^1$ two edges such that $\pi(g_1)=\pi(g_2)$. From $\vartheta'(g_1)=\vartheta'(g_2)$ we conclude  that 
$\vartheta'(t_G(g_1))=\vartheta'(t_G(g_2))$, and from the injectivity on vertices we deduce that $t_{G}(g_1)=t_{G}(g_2)$. Finally, since $\varphi'(g_1)=\varphi'(g_2)$, 
it follows from target injectivity that $g_1=g_2$.
\end{proof}

To end this section, note that  Example~\ref{ex:counter} also shows that not every $R\in\mathsf{PRG}(E,F)$ enjoys a pushout-type decomposition.

\section{Cohn path algebras and Toeplitz graph C*-algebras}\label{sec:5}
\noindent
The next step is to determine a subcategory of $\mathsf{PRG}$ admitting a functor to $k$-$\mathsf{Alg}$ by assigning Cohn path algebras to graphs
and algebra homomorphisms to relation morphisms. Unlike in the preceding section, this time we can construct a subcategory of 
$\mathsf{PRG}$ that contains both $\mathsf{MIPG}^{\mathrm{op}}$ and $\mathsf{TBPOG}$. Completing Cohn path algebras to Toeplitz graph C*-algebras
automatically yields a functor to the category C*-$\mathsf{Alg}$ of C*-algebras and *-homomorphisms.

\subsection{The subcategory $\mathsf{CRG}$ and the Cohn--Toeplitz functor}

To induce algebra homomorphisms between Cohn path algebras, and subsequently *-homomorphisms between Toeplitz graph C*-algebras,
we need conditions ensuring that the first Cuntz--Krieger relation (CK1) (see Definition~\ref{ck1})
is respected.
\begin{definition}
Given a relation morphism $R:E\to F$, we say that it is \emph{target surjective}
if the restriction $t_{PE}:R^{-1}(f)\to R^{-1}(t_F(f))$
is surjective for all $f\in F^1$. We say that $R$ is \emph{monotone}
whenever $(x,f),(x',f')\in G_R^1$ and $x\preceq x'$ imply that $(x,f)=(x',f')$.
\end{definition}

Now we are ready to prove the composability lemma:
\begin{lemma}\label{lem:composition-TS-M}
Let $R\in\mathsf{PRG}(E,F)$ and $S\in \mathsf{PRG}(F,G)$. Then,
\begin{enumerate}
\item 
if $R$ and $S$ are target surjective, then $S\circ R$ is target surjective;
\item 
if $R$ and $S$ are monotone, then $S\circ R$ is monotone.
\end{enumerate}
\end{lemma}

\begin{proof}
(1) Let $u\in E^0$ and $g\in G^1$ be such that $(u,t_G(g))\in S\circ R$. Since the preimages of  relations preserve vertices, by the definition of composition, there 
exists $v\in F^0$ such that $(u,v)\in R$ and $(v,t_G(g))\in S$. Also, since $S$ is target surjective, there exists $y\in  \mathrm{FP}(F)$ such that $t_{PF}(y)=v$ 
and $(y,g)\in S$. 
If $y\in F^0$, then $y=t_{PF}(y)=v$. In this case, $(u,g)\in S\circ R$ and $t_E(u)=u$.

Now, assume that $y=f_1\ldots f_n$, where  $f_1,\ldots,f_n\in F^1$, for some $n\geq 1$. Note that $t_F(f_n)=t_{PF}(y)=v$. Since $R$ is target surjective, there exists 
$x_n\in  \mathrm{FP}(E)$ such that $(x_n,f_n)\in R$ and $t_{PE}(x_n)=u$. Recursively, for $i=1,\ldots,n-1$, we find $x_i$ from $x_{i+1}$ as follows.
Given $(x_{i+1},f_{i+1})\in R$, 
we note that $(s_{PE}(x_{i+1}),s_F(f_{i+1}))\in R$. Since $t_F(f_i)=s_F(f_{i+1})$ and $R$ is target surjective, there exists 
$x_i\in  \mathrm{FP}(E)$ such that $(x_i,f_i)\in R$ and 
$t_{PE}(x_i)=s_{PE}(x_{i+1})$. Put $x:=x_1\ldots x_n$. Then $t_{PE}(x)=t_{PE}(x_n)=u$, and the multiplicativity of $R$ implies that $(x,y)\in R$. Hence, 
$(x,g)\in S\circ R$, so $S\circ R$ is target surjective.

(2) Let $x,x'\in  \mathrm{FP}(E)$ and $g,g'\in G^1$ be such that $(x,g),(x',g')\in S\circ R$ and $x\preceq x'$. 
By definition, there exist $y,y'\in  \mathrm{FP}(F)$ such that $(x,y),(x',y')\in R$ 
and $(y,g),(y',g')\in S$.
First, assume  that $y'\in F^0$. In this case, $x'\in E^0$ as the preimage of $R$ preserves vertices. Furthermore, as $x\preceq x'$ implies that $x=x'$,
we conclude that 
$(x',y)\in R$. Hence, $s_{PF}(y)=y'$ because $R$ preserves sources and is vertex disjoint. Therefore, $y'\preceq y$, and it follows from
the monotonicity of $S$ that $g=g'$.

Now, assume that $y\in F^0$. Observing that $s_{PE}(x)=s_{PE}(x')$, a similar argument as above, shows that $s_{PF}(y')=y$. In particular, we know
that $y\preceq y'$. Therefore,
since $S$ is monotone, we infer that $y'=y\in F^0$, which brings us back to the previous case.

Finally, assume that $y=f_1\cdots f_n$ and $y'=f'_1\cdots f'_m$, where  \mbox{$f_1,\ldots,f_n,f'_1,\ldots,f'_m\in F^1$}
and $m,n\geq 1$. By decomposability, there are 
$x_1,\ldots,x_n,x'_1,\ldots,x'_m\in  \mathrm{FP}(E)$ such that $x=x_1\ldots x_n$, $x'=x'_1\ldots x'_m$, $(x_i,f_i)\in R$ for all $i=1,\ldots,n$ and $(x'_j,f'_j)\in R$ for all 
$j=1\ldots,m$. Since $x\preceq x'$, we note  that either $x_1\preceq x'_1$ or $x_1'\preceq x_1$, so, from the montonicity of $R$, we conclude that $f_1=f'_1$ and 
$x_1=x'_1$. In this case, $x_2\ldots x_n\preceq x'_2\ldots x'_m$. Repeating the argument $\min\{m,n\}$ times, we see that either $y\preceq y'$ or $y'\preceq y$. In 
any  case, since $S$ is monotone, we conclude that $g=g'$ and $y=y'$. The latter implies that $m=n$, so  $x=x'$.
\end{proof}

As it is clear that the identity morphisms satisfy both target surjectivity and monotonicity, we thus obtain the domain category for the Cohn--Toeplitz functor:
\begin{theorem}
Relation morphisms in $\mathsf{PRG}$ that are target surjective and monotone define a subcategory whose objects are all directed graphs. (We denote this 
subcategory by $\mathsf{CRG}$.) 
\end{theorem}

To obtain the functor, first we prove:
\begin{lemma}\label{lem:RC}
	Let $R:E\to F$ be a relation morphism in~$\mathsf{CRG}$.
	Then the formulas
	\begin{equation*}
		R_C^*(S_y):= \sum_{x\in R^{-1}(y)}S_x,\qquad\quad
		R_C^*(S_{y^*}):= \sum_{x\in R^{-1}(y)}S_{x^*},
	\end{equation*}
	where $y\in F^0\cup F^1$ and $y^*:=y$ for $y\in F^0$, define an algebra homomorphism $R_C^*:C_k(F)\to C_k(E)$.
\end{lemma}
\begin{proof}
	As in the proof of Lemma~\ref{lem:RP}, we can easily see that $R_C^*$ satisfies (V), (E1) and~(E2). It remains to prove that $R_C^*$ satisfies~(CK1). 
	To this end, take  $f,f'\in F^1$. First, assume that $f\neq f'$. In this case,
	\[
	R_C^*(S_{f^*}) R_C^*(S_{f'})=\sum_{x\in R^{-1}(f)}\ \sum_{x'\in R^{-1}(f')}S_{x^*}S_{x'} =0.
	\]
Indeed,	it follows from the monotonicity of $R$ that ${x}\neq{x'}$, so  all terms in the above sum vanish. 
	
	Now, assume that $f=f'$. Using the fact that $R$ is both monotone and target bijective, we compute that
\begin{align}
		R_C^*(S_{f^*}) R_C^*(S_f)&=\sum_{x\in R^{-1}(f)}\ \sum_{x'\in R^{-1}(f')}S_{x^*}S_{x'}
		=\sum_{x\in R^{-1}(f)}S_{x^*}S_{x}\nonumber\\
		&=\sum_{x\in R^{-1}(f)}S_{t_{PE}(x)}
		=\sum_{u\in R^{-1}(t_F(f))}S_u
		=R_C^*(S_{t_F(f)}),
\end{align}
as needed.
\end{proof}

Combining Lemma~\ref{lem:RC} with a functoriality argument as in Theorem~\ref{thm:functor.PRG}, we can now claim the main result of this section:
\begin{theorem}\label{thm:functor.CRG}
The association of Cohn path algebras to graphs and algebra homomorphisms $R_C^*$ to relation morphisms $R$ 
defines a contravariant functor from $\mathsf{CRG}$ to $k$-$\mathsf{Alg}$.
\end{theorem}

As the Toeplitz graph C*-algebras are the universal enveloping C*-algebras of the Cohn path *-algebras, we  immediately obtain:
\begin{corollary}
The association of Toeplitz graph C*-algebras to graphs and *-ho\-mo\-mor\-phisms $R_C^*$ to relation morphisms $R$ induces a contravariant functor 
from $\mathsf{CRG}$ to \textup{C*-}$\mathsf{Alg}$.
\end{corollary}

\subsection{Subcategories and factorizations}

\begin{proposition}\label{prop:embedding.Cohn}
The embedding of $\mathsf{OG}$ into $\mathsf{RG}$ induces an embedding of $\mathsf{TBPOG}$ into $\mathsf{CRG}$.
The embedding of $\mathsf{PG}^{\mathrm{op}}$ into $\mathsf{RG}$ induces an embedding of $\mathsf{MIPG}^{\mathrm{op}}$ into $\mathsf{CRG}$.
\end{proposition}
\begin{proof}
Let $\varphi\in\mathsf{TBPOG}(E,F)$. We claim that $R_{\varphi}$ is monotone. To show this, take
$(x,f),(x',f')\in R_{\varphi}$ such that $x\preceq x'$. Then $x,x'\in E^1$, so  $x=x'$. Hence, $f=\varphi(x)=\varphi(x')=f'$. 
Furthermore, the target-surjectivity of $R_{\varphi}$ is easily  seen as equivalent to the target-surjectivity of~$\varphi$.

Next, let $\vartheta\in\mathsf{MIPG}^{\mathrm{op}}(E,F)$. We see that $R^{\vartheta}$ is target surjective because $t_{PE}(\vartheta(f))=\vartheta(t_F(f))$ for all 
$f\in F^1$. Finally, if ${\vartheta}$ is monotone, then $f,f'\in F^1$ and $\vartheta(f)\preceq\vartheta(f')$ imply that $f=f'$. Hence, also
$(\vartheta(f),f)=(\vartheta(f'),f')$, so $R_{\vartheta}$ is monotone.
\end{proof}

\begin{proposition}\label{prop:unique2}
Every $R\in\mathsf{CRG}(E,F)$ admits a factorization
$R=R_\varphi\circ R^\vartheta$, where $\vartheta\in\mathsf{MIPG}(G,E)$ and $\varphi\in\mathsf{TBPOG}(G,F)$. 
This factorization is unique (up to a canonical isomorphism), and given explicitly by the construction in Proposition~\ref{prop:factor}.
\end{proposition}
\begin{proof}
It follows from Proposition~\ref{prop:factunique}  that the factorization  $R=R_\varphi\circ R^\vartheta$ in Proposition~\ref{prop:factor}
is unique with $\vartheta\in\mathsf{IPG}(G_R,E)$ 
and target-injective $\varphi\in\mathsf{POG}(G_R,F)$. It remains to show that 
$\varphi$ 
is also target surjective and that $\vartheta$ is also monotone.
To this end, take
$f\in F^1$, and note that $\varphi^{-1}(t_F(f))=\{(u,t_F(f))\in R\mid u\in E^0\}$ and $\varphi^{-1}(f)=\{(x,f)\in R\mid x\in  \mathrm{FP}(E)\}$. 
Now, the map
\[
\varphi^{-1}(f)\ni(x,f)\longmapsto (t_{PE}(x),t_F(f))\in \varphi^{-1}(t_F(f))
\]
is bijective because $R$ is target bijective. Therefore, also $\varphi$ is target bijective.
Finally,
by the monotonicity of $R$, if $(x,f),(x',f')\in G_R^1$ are such that $\vartheta(x,f)\preceq\vartheta(x',f')$, then $(x,f)=(x',f')$. Hence, $\vartheta$ is monotone.
\end{proof}

Note that, as  in  previous sections, Example~\ref{ex:counter} also shows that not every $R\in\mathsf{CRG}(E,F)$ has a pushout-type decomposition.

\section{Leavitt path algebras and graph C*-algebras}\label{sec:6}
\noindent
Finally, we determine a subcategory of $\mathsf{CRG}$ admitting a functor to $k$-$\mathsf{Alg}$ by assigning Leavitt path algebras to graphs and algebra 
homomorphisms to relation morphisms. This subcategory will contain both $\mathsf{RMIPG}^{\mathrm{op}}$ and $\mathsf{CRTBPOG}$.
As in the preceding section, completing Leavitt path algebras to  graph C*-algebras
automatically yields a functor to the category C*-$\mathsf{Alg}$ of C*-algebras and *-homomorphisms.

\subsection{The category $\mathsf{ARG}$ of admissible relation morphisms of graphs and the Leavitt functor}

The final condition to arrive at our main-destination category is:
\begin{definition}
Given a relation morphism $R:E\to F$, we say that it is \emph{regular}
whenever $v\in\reg(F)$, $u\in R^{-1}(v)$ and $x\in s_{PE}^{-1}(u)$
imply that there is $(x',f)\in G_R^1$ such that $s_F(f)=v$ and $x\sim x'$.
\end{definition}
For starters, we need the following technical result:
\begin{lemma}\label{continue}
Assume that $R\in\mathsf{CRG}(E,F)$ is regular, $f\in s_F^{-1}(\reg(F))$ and $x,x'\in  \mathrm{FP}(E)$. 
Then  $(x,f)\in R$ and $x'\prec x$ imply that $t_{PE}(x')\in \reg(E)$.
\end{lemma}
\begin{proof}
Denote $v:=s_F(f)$ and $u:=t_{PE}(x')$.
Since $x'\prec x$, we know that
$u$ is not a sink. Suppose now that $u$ is an infinite emitter. Then, as $R$ is proper and regular, for some $e\in s_E^{-1}(u)$, there are 
$\ol{f}\in s_F^{-1}(v)$ and $\ol{x}\in  \mathrm{FP}(E)$ such that $(\ol{x},\ol{f})\in R$ and $\ol{x}\prec x'e$. Hence, $\ol{x}\preceq x'\prec x$, so
$\ol{x}\neq x$. However, this contradicts the monotonicity of $R$ which implies that $f=\ol{f}$ and $x=\ol{x}$.
\end{proof}

Given a graph $E$, let $I_E$ be the ideal in $C_k(E)$ generated by the set~\eqref{en:Leavitt}. Our key lemma is:
\begin{lemma}\label{regularity}
Let $R\in\mathsf{CRG}(E,F)$. Then the map $R_C^*:C_k(F)\to C_k(E)$
satisfies $R_C^*(I_F)\subseteq I_E$ if and only if $R$ is regular.
\end{lemma}
\begin{proof}
    $(\Rightarrow)$ Assume that $R_C^*(I_F)\subseteq I_E$. This implies that $R_C^*$  descends to the quotient homomorphism $R_L^*:L_k(F)\to L_k(E)$. 
   Next, let $v\in \reg(F)$, $\ol{u}\in E^0$, $\ol{x}\in  \mathrm{FP}(E)$ be such that $(\ol{u},v)\in R$ and $s_{PE}(\ol{x})=\ol{u}$. On the one hand,
    \[
    R_L^*(S_v)=\sum_{u\in R^{-1}(v)}S_u.
    \]
    On the other hand
    \[
    R_L^*(S_v)=R_L^*\left(\sum_{f\in s_F^{-1}(v)}S_fS_{f^*}\right)=\sum_{f\in s_F^{-1}(v)}\ \sum_{x\in R^{-1}(f)}S_xS_{x^*}.
    \]
    Now, since $R_L^*(S_v)S_{\ol{x}}=S_{\ol{u}}S_{\ol{x}}=S_{\ol{x}}\neq 0$, there must exist $f\in s_F^{-1}(v)$ and 
    $x\in  \mathrm{FP}(E)$ such that $S_xS_{x^*}S_{\ol{x}}\neq 0$. This implies that $x\sim \ol{x}$, so $R$ is regular.
    
    $(\Leftarrow)$ Assume that $R$ is regular. Proving that $R_C^*(I_F)\subseteq I_E$ is equivalent to proving that
    \begin{equation}\label{CK2.part1}
        \sum_{u\in R^{-1}(v)}S_u=\sum_{f\in s^{-1}(v)}\ \sum_{x\in R^{-1}(f)}S_xS_{x^*}
    \end{equation}
    holds in $L_k(E)$ for every $v\in \reg(F)$. To prove the latter, take $v\in\reg(F)$. From now on, our formulas are assumed to be in Leavitt path algebras. 
    Note first that, if $x\in E^0$, then the summand 
    $S_xS_{x^*}$ is equal to $S_x$. Therefore, we can  cancel in \eqref{CK2.part1} all terms of the form $S_u$, 
    where $u\in E^0$ is such that $(u,v)\in R$ and $(u,f)\in R$ 
    for some $f\in s_F^{-1}(v)$. Now, if $S_u$ in \eqref{CK2.part1} is not canceled, then $u$ is a regular vertex due to
    the regularity of $R$ and Lemma~\ref{continue}. Hence, for every not-canceled $S_u$, we can write
    \begin{equation}\label{CK2.part2}
        S_u=\sum_{e_1\in s_E^{-1}(u)}S_{e_1}S_{e_1^*}.
    \end{equation}
    Using again the regularity of $R$ and Lemma~\ref{continue}, we infer that, if $t_E(e_1)$ is not regular, then there exists $f\in s_F^{-1}(v)$ such that
     $(e_1,f)\in R$. Next, using the right-hand side of~\eqref{CK2.part1}, we cancel all $S_{e_1}S_{e_1^*}$ from \eqref{CK2.part2} substituted to \eqref{CK2.part1}
     for all $e_1$ such that $(e_1,f)\in R$ for some 
     $f\in s_F^{-1}(v)$. Arguing as above, if $S_{e_1}S_{e_1^*}$  remains, then $t_E(e_1)\in \reg(F)$, so we can write
    \begin{equation}\label{CK2.part3}
        S_{e_1}S_{e_1^*}=\sum_{e_2\in s_E^{-1}(t_E(e_1))}S_{e_1e_2}S_{(e_1e_2)^*}.
    \end{equation}
    We continue the process of cancellation
    until we get paths of length $n$, where 
    \[
    n:=\max\{|x|\;|\;(x,f)\in R\text{ for some }f\in s_F^{-1}(v)\}.
    \]

    It remains to show that this process cancels all summands in~\eqref{CK2.part1}. At every step $0\leq k\leq n$ of the cancellation process,
     if the term $S_{e_1\cdots e_k}S_{(e_1\cdots e_k)^*}$ survives (when $k=0$, we mean the projections $S_u$), 
    by the 
    regularity of $R$, there exist $f\in s_F^{-1}(v)$ and $x\in  \mathrm{FP}(E)$ such that $(x,f)\in R$ and $e_1\cdots e_k\prec x$. This implies that by expanding the 
    left-hand side of \eqref{CK2.part1} using (CK2), at the step $n$, all elements of these expansions will be canceled. Finally, to show that that no term remains 
    on the right-hand side of \eqref{CK2.part1}, take $x\in  \mathrm{FP}(E)$ such that $(x,f)\in R$ for some $f\in s_F^{-1}(v)$. By the monotonicity of $R$, no
    initial sub-path of $x$ can be of the form
$e_1\cdots e_k$, where $k<|x|$. Now, by Lemma~\ref{continue}, the end vertex of the $k$-th egde of $x$ is regular for all $0\leq k<|x|$. 
This implies that $S_xS_{x^*}$ appears in the step $|x|$ of the cancellation process,
   so it is removed at this step.
\end{proof}

The composability of the regularity condition is an easy consequence of the above lemma.
\begin{corollary}\label{lem:composition.Reg}
Let $R\in\mathsf{CRG}(E,F)$ and $S\in\mathsf{CRG}(F,G)$. Then, if $R$ and $S$ are regular, so is~$S\circ R$.
\end{corollary}
\begin{proof}
By Lemma~\ref{regularity}, $R_C^*(I_F)\subseteq I_E$ and $S_C^*(I_G)\subseteq I_F$. Hence,
\[
(S\circ R)_C^*(I_G)=R_C^*(S^*(I_G))\subseteq R_C^*(I_F)\subseteq I_E.
\]
Applying now Lemma~\ref{regularity} in the opposite direction, we conclude that $S\circ R$ is regular.
\end{proof}

We put all above conditions on relation morphisms into a single definition to obtain a key concept of our paper:
	\begin{definition}
	A relation morphism $R:E\to F$ is said to be \emph{admissible} if:
	\begin{enumerate}
		\item it is multiplicative, that is, $(x,y),(x',y')\in R$, $s_{PE}(x')=t_{PE}(x)$ and $s_{PF}(y')=t_{PF}(y)$ imply that $(xx',yy')\in R$;
		\item it is decomposable, that is, $y,y'\in  \mathrm{FP}(F)$, $s_{PF}(y')=t_{PF}(y)$ and $(\ol{x},yy')\in R$ imply that 
		$\exists\;x\in R^{-1}(y)$, $x'\in R^{-1}(y'):xx'=\ol{x}$;
		\item it is proper, that is, $|R^{-1}(y)|<\infty$ for all $y\in  \mathrm{FP}(F)$;
		\item it is vertex disjoint, that is, $v,\,v'\in F^0$ and $v\neq v'$ imply $R^{-1}(v)\cap R^{-1}(v')=\emptyset$;
		\item it is target bijective, that is, the restriction $t_{PE}:R^{-1}(f)\to R^{-1}(t_F(f))$ is bijective for all $f\in F^1$;
		\item it is monotone, that is, $(x,f),(x',f')\in G_R^1$ and $x\preceq x'$ imply that $(x,f)=(x',f')$;
		\item it is regular, that is, $v\in\reg(F)$, $u\in R^{-1}(v)$ and $x\in s_{PE}^{-1}(u)$
		imply that there is $(x',f)\in G_R^1$ such that $s_F(f)=v$ and $x\sim x'$.
	\end{enumerate}
\end{definition}

As the identity morphisms are regular,  Corollary~\ref{lem:composition.Reg} implies the following crucial result:
\begin{theorem}
Admissible relation morphisms define a subcategory of $\mathsf{RG}$ whose objects are all directed graphs. (We denote this subcategory by $\mathsf{ARG}$.) 
\end{theorem}

We can now claim the main theorem of the paper:
\begin{theorem}\label{main}
The association of Leavitt path algebras to graphs and algebra homomorphisms $R_L^*$ to admissible relation morphisms $R$ defines a contravariant functor 
from $\mathsf{ARG}$ to $k$-$\mathsf{Alg}$.
\end{theorem}
\begin{proof}
 The fact that $R_C^*:C_k(F)\to C_k(E)$ descends to the quotient algebra homomorphism $R_L^*:L_k(F)\to L_k(E)$ follows from Lemma~\ref{regularity}. 
 The functoriality of this 
 association then follows from Theorem~\ref{thm:functor.CRG}.
\end{proof}

As the  graph C*-algebras are the universal enveloping C*-algebras of the Leavitt path *-algebras, we  immediately obtain:
\begin{corollary}
The association of  graph C*-algebras to graphs and *-homomorphisms $R_L^*$ to admissible relation morphisms $R$ induces a contravariant functor 
from $\mathsf{ARG}$ to \textup{C*-}$\mathsf{Alg}$.
\end{corollary}

\subsection{Subcategories and factorizations}

This time our result on subcategories is slightly stronger, i.e.\ we have equivalences instead of implications.
\begin{proposition}\label{lem:co-contra-regularity}
Let $E$ and $F$ be any graphs and let $\varphi\in\mathsf{TBPOG}(E,F)$ and $\vartheta\in\mathsf{MIPG}(F,E)$. Then
\begin{enumerate}
\item
$\varphi\in\mathsf{CRTBPOG}(E,F)$ if and only if $R_\varphi\in\mathsf{ARG}(E,F)$,
\item
$\vartheta\in\mathsf{RMIPG}(F,E)$ if and only if
$R^\vartheta\in\mathsf{ARG}(E,F)$.
\end{enumerate}
\end{proposition}
\begin{proof}
(1) If $\varphi\in\mathsf{TBPOG}(E,F)$, then $R_\varphi\in \mathsf{CRG}(E,F)$ by Proposition~\ref{prop:embedding.Cohn}.
Assume first that $\varphi^{-1}(\reg(F))\subseteq \reg(E)$, and take $u\in E^0$ and $x\in  s_{PE}^{-1}(u)$ such that 
$v:=\varphi(u)\in \reg(F)$. Then $u\in \reg(E)$. If $x\in E^0$, then $x=u$. If $e\in s_E^{-1}(u)$, which exists because 
$u\in \reg(E)$, then $(e,\varphi(e))\in R_{\varphi}$ and $x\preceq e$. If $|x|\geq 1$, we take $e$ as the first edge of~$x$. 
Then $(e,\varphi(e))\in R_{\varphi}$ and $e\preceq x$. In both cases, $\varphi(e)\in s_F^{-1}(v)$ and $e\sim x$. Hence, $R_{\varphi}$ is regular.

Assume now that $R_{\varphi}$ is regular, and take $v\in \reg(F)$. We have to prove that if $u\in E^0$ is such that $\varphi(u)=v$, then $u \in \reg(E)$. 
From the 
regularity of $R_{\varphi}$ and the fact that $u\preceq u$,  we infer that $s_E^{-1}(u)\neq\emptyset$. 
Also, note that $s_E^{-1}(u)\subseteq \varphi^{-1}(s_F^{-1}(v))$. Furthermore, since $R_{\varphi}$ is proper and $s_F^{-1}(v)$ is finite, $s_E^{-1}(u)$ is finite. 
Therefore, $u\in \reg(E)$.

(2) If $\vartheta\in\mathsf{MIPG}(F,E)$, then  $R^{\vartheta}\in\mathsf{CRG}(E,F)$ by Proposition~\ref{prop:embedding.Cohn}.
Assume first that $\vartheta$ is regular, and take $v\in \reg(F)$, $u=\vartheta(v)$ and $x\in  \mathrm{FP}(E)$ such that $s_{PE}(x)=u$. We have to find 
$x'\in  \mathrm{FP}(E)$ such that $x'=\vartheta(f)$ for some $f\in s_F^{-1}(v)$ and $x\sim x'$. We divide the quest into three cases.

Case 1: $|x|=0$. In this case,  $x=u$. Take any $f\in s_F^{-1}(v)$, which exists because $v\in \reg(F)$, and let $x'=\vartheta(f)$. Then $(x',f)\in R^{\vartheta}$ and 
$s_{PE}(x')=u=x$, so $x\preceq x'$.

Case 2: $x=e_1\ldots e_k$ for some $k\geq 1$, $e_1,\ldots,e_k\in E^1$ and $v$ satisfies the condition (3)(1) of Definition~\ref{covadm}.
Now, let $\{f_1,\ldots,f_m\}:=s_F^{-1}(v)$ 
and $x_i:=\vartheta(f_i)$ for $i=1,\ldots,m$. If $e_1=x_i$ for some $i\in \{1,\ldots,n\}$, take $x'=x_i$, and we are done. Otherwise,  consider the set 
$I_1=\{i\in\{1,\ldots,n\}\mid e_1\preceq x_i\}$. This set is nonempty by the condition (3)(1)(b)(iii) of Definition~\ref{covadm} because $s_E(e_1)=u$. 
If $e_1e_2=x_i$ for some 
$i\in I_1$, we take $x'=x_i$, and again we are done. Otherwise, consider $I_2=\{i\in I_1\mid e_1e_2\preceq x_i\}$. This set is nonempty by the condition
 (3)(1)(b)(iii) of 
Definition~\ref{covadm} because, for every $i\in I_1$, the source of the second edge of $x_i$ coincides with $t_E(e_1)=s_E(e_2)$. 
We continue the process this way. Then, we 
either find $1\leq l\leq k$ such that $x_i=e_1\cdots e_l$ and  take $x=x_i$, or, using the condition (3)(1)(b)(iii) of Definition~\ref{covadm} at the last step, we find 
$x_i$ such that $x=e_1\cdots e_k\prec x_i$, and  take $x'=x_i$. Hence, $R^\vartheta$ is regular, whence admissible.

Case 3: $v\in\reg_0(F)$ and $\vartheta(s_F^{-1}(v))=\vartheta(v)$. In this case, we can take $x'=u$ because $(u,f)\in R^{\vartheta}$ for the single edge 
$f\in s_F^{-1}(v)$.

Assume now that $R^{\vartheta}$ is admissible. In particular, it is regular. 
Let $v\in \reg(F)$ and $u= \vartheta(v)$. We need to consider two cases. In the first case, we assume that there exists 
$f\in s_F^{-1}(v)$ such that $\vartheta(f)=u$. Then $f$ is a loop based at $v$, and the monotonicity of $\vartheta$ impies that $v\in \reg_0(E)$, which shows
that $\vartheta$ satisfies the
condition  (3)(2) in Definition~\ref{covadm}. 

In the second case, we assume that there does not exist $f$ as above. Then $ \vartheta(f)$ has 
positive length for all $f\in s_F^{-1}(v)$. Write $s_F^{-1}(v)=:\{f_1,\ldots,f_m\}$ and $x_i:= \vartheta(f_i)$ for $i=1,\ldots,m$. 
Note that, in the case we are considering 
now, $|x_i|\geq 1$ for all $i=1,\ldots,m$. Finally, observe that the monotonicity of $R^{\vartheta}$ implies that the restriction of $\vartheta$ to
$s_F^{-1}(v)$ yields  a bijection between $\{f_1,\ldots,f_m\}$ 
and $\{x_1,\ldots,x_m\}$, so $\vartheta$ satisfies the condition (3)(1)(a) of Definition~\ref{covadm}.

We now prove that $\vartheta$ satisfies the condition (3)(1)(b) of Definition~\ref{covadm}. Assume  first that $x=x_i$ for some $i=1,\ldots,m$. 
As mentioned above, $n:=|x_i|\geq 1$, so 
$x_i=e_1\ldots e_n$ for some $e_1,\ldots,e_n\in E^1$, and $x$ satisfies the condition (3)(1)(b)(i). Next, if $xx'=x_j$ for some $j=1,\ldots,m$, the monotonicity of 
$\vartheta$ implies that $i=j$ and $x'=t_{PE}(x)$, so  $x$ satisfies the condition (3)(1)(b)(ii).
Finally, let 
$l\in \{1,\ldots,n\}$, $e\in s_E^{-1}(s_E(e_l))$ and $\ol{x}=e_1\cdots e_{l-1}e$. Note that $s_{PE}(\ol{x})=u$. By the regularity of $R^{\vartheta}$, there exists 
$j\in\{1,\ldots,m\}$ such that $x_j\sim \ol{x}$. 
Since  the monotonicity of $R^{\vartheta}$ excludes $x_j\prec\ol{x}$, we conclude that $x_j=\ol{x}x''$ for some 
$x''\in  \mathrm{FP}(E)$. This proves that also the condition (3)(1)(b)(iii) is satisfied by~$x$.
Assume now that $x\in  \mathrm{FP}(E)$ satisfies the three conditions (3)(1)(b)(i)--(3)(1)(b)(iii). By (3)(1)(b)(i), $x$ is not a vertex. Let $e$ be the first edge of $x$. By 
(3)(1)(b)(iii), there exists $x'\in FP(E)$ such that $ex'\in \vartheta(s_F^{-1}(v))$, that is, there exists $f\in s_F^{-1}(v)$ such that $\vartheta(f)=ex'$. Then
\[
s_{PE}(x)=s_E(e)=s_{PE}(ex')=s_{PE}(\vartheta(f))=\vartheta(s_F(f))=\vartheta(v)=u. 
\]
Furthermore, by the regularity of $R^\vartheta$, there exists $i\in\{1,\ldots,m\}$ such that 
$x\sim x_i$. However, the monotonicity of $R^{\vartheta}$ excludes $x_i\prec x$ by (3)(1)(b)(iii), so $x\preceq x_i$. Hence,
$x=x_i$ by (3)(1)(b)(ii).  
\end{proof}

The above proposition combined with Proposition~\ref{prop:embedding.Cohn} yields:
\begin{corollary}
The embedding of $\mathsf{OG}$ into $\mathsf{RG}$ induces an embedding of $\mathsf{CRTBPOG}$ into $\mathsf{ARG}$.
The embedding of $\mathsf{PG}^{\mathrm{op}}$ into $\mathsf{RG}$ induces an embedding of $\mathsf{RMIPG}^{\mathrm{op}}$ into $\mathsf{ARG}$.
\end{corollary}

Unlike for path algebras and Cohn path algebras, the factorization claim for Leavitt path algebras requires the target graph to be row finite.
The fact that this additional assumption is needed is shown in Example~\ref{ex:nopull}.
\begin{proposition}\label{rowf}
Let $E$ be any graph and $F$ be a row-finite graph. Then
every admissible relation morphism $R\in\mathsf{ARG}(E,F)$ admits a factorization
$R=R_\varphi\circ R^\vartheta$, where $\vartheta\in\mathsf{RMIPG}(G,E)$ and $\varphi\in\mathsf{CRTBPOG}(G,F)$. This factorization is unique (up to a canonical 
isomorphism), and given explicitly by the construction in Proposition~\ref{prop:factor}.
\end{proposition}
\begin{proof}
It follows from Proposition \ref{prop:unique2} that $R$ has a unique factorization $R=R_\varphi\circ R^\vartheta$, where $\vartheta\in\mathsf{MIPG}(G_R,E)$ and 
$\varphi\in\mathsf{TBPOG}(G_R,F)$. Now
we need to prove that $\varphi$ and $\vartheta$ are regular.

First we show that $\varphi$ is regular. Let $(u,v)\in \varphi^{-1}(\reg(F))$, that is, $v\in\reg(F)$. Since $s_{PE}(u)=u$, by the regularity of $R$, there exist 
$f\in s_F^{-1}(v)$ and $x\in  \mathrm{FP}(E)$ such that $(x,f)\in R$ and $x\sim u$. Note that the last condition is the same as $s_{PE}(x)=u$. 
Hence, $(x,f)\in G_R^1$ is 
such that $s_{G_R}(x,f)=(u,v)$ and $(u,v)$ is not a sink. Furthermore, since $s_F^{-1}(v)$ is finite and $R$ is proper, we infer  that $s_{G_R}^{-1}(u,v)$ is finite.
Therefore, $(u,v)\in\reg(G_R)$.

 By Proposition~\ref{lem:co-contra-regularity}, to show the regularity of $\vartheta$ it suffices to prove  that $R^{\vartheta}$ is regular. 
To this end, take $(u,v)\in \reg(G_R)$ and 
$x\in  \mathrm{FP}(E)$  such that $s_{PE}(x)=u$. Now we have to find $f\in s_F^{-1}(v)$ and $x'\in  \mathrm{FP}(E)$ such that $(x',f)\in R$ and $x\sim x'$. 
We see from the 
definition of $s_{G_R}$ that $v$ is not a sink. Hence, since $F$ is row-finite, $v$ is regular. Now the existence of $f$ and $x'$ follows from the regularity of~$R$.
\end{proof}

Let us now consider an example of transitively closed admissible relations composing to an admissible relation that is not transitively closed and whose 
transitive closure is not admissible. This example shows that the transitive closure is at odds with admissibility, which indicates that Proposition~\ref{criterion}
is a serious obstruction for pushout-type factorizations of admissible relation morphisms.
\begin{example}\footnote{We are grateful to Jack Spielberg for discussions concerning this example.}
Consider the relation morphism $R:=R_\varphi\circ R^{\vartheta}$ depicted below.
\begin{center}
\begin{tikzpicture}[xscale=3,yscale=2]

\node[draw,rectangle,color=gray!60] (a1) at (0,0) {\includegraphics[page=38,valign=c]{images.pdf}};
\node[draw,rectangle,color=gray!60] (a2) at (1.5,1.5) {\includegraphics[page=39,valign=c]{images.pdf}};
\node[draw,rectangle,color=gray!60] (a3) at (3,0) {\includegraphics[page=40,valign=c]{images.pdf}};

\path[-To,font=\footnotesize,inner sep=1pt,shorten <=2pt,shorten >=2pt]
(a2) edge node[above left] {$\vartheta$} (a1)
(a2) edge node[above right] {$\varphi$} (a3)
(a1) edge[dashed] node[above] {$R$} (a3);

\draw
(a1.south) node[below] {$E$}
(a2.east) node[right] {$G$}
(a3.south) node[below] {$F$};

\end{tikzpicture}
\end{center}
Here the map $\vartheta$ is given by $\vartheta(w):=w$, $\vartheta(e):=e$ and $\vartheta(f):=v=:\vartheta(v)$, and the map $\varphi$ is given by 
$\varphi(v):=u=:\varphi(w)$ and $\varphi(e):=h=:\varphi(f)$. One can easily verify that 
$\vartheta\in \mathsf{RMIPG}$ and $\varphi\in\mathsf{CRTBPOG}$, so
$R^\vartheta$ and $R_\varphi$ are admissible and transitively closed. Since the composition preserves 
admissibility, we infer  that 
\[
R=\big\{(v,u),(v,h^n),(w,u),(e^n,h^n)\mid n\in\mathbb{N}\setminus\{0\}\big\}
\]
is admissible.
\begin{center}
\begin{tikzpicture}[baseline=(current bounding box.center),semithick,font=\small,scale=0.7]

\draw[-To] (-0.5,0) -- (8.5,0) node[below] {$\mathrm{FP}(E)$};
\draw[-To] (0,-0.5) -- (0,6.5) node[left] {$\mathrm{FP}(F)$};

\filldraw (0,0) circle (0.05);

\foreach \x in { 1,2,3,4,6} \draw (\x,0.1) -- (\x,-0.1);
\foreach \y in { 1,2,3,5} \draw (0.1,\y) -- (-0.1,\y);

\node[above] at (1,-0.8) {$v$};
\node[above] at (2,-0.8) {$w$};
\node[above] at (3,-0.8) {$e$};
\node[above] at (4,-0.8) {$e^2$};
\node[above] at (5,-0.8) {$\cdots$};
\node[above] at (6,-0.8) {$e^n$};
\node[above] at (7,-0.8) {$\cdots$};

\node[right] at (-1,1) {$u$};
\node[right] at (-1,2) {$h$};
\node[right] at (-1,3) {$h^2$};
\node[right] at (-0.85,4.1) {$\vdots$};
\node[right] at (-1,5) {$h^n$};
\node[right] at (-0.85,5.9) {$\vdots$};

\filldraw (1,1) circle (0.05);
\filldraw (1,2) circle (0.05);
\filldraw (1,3) circle (0.05);
\filldraw (1,5) circle (0.05);
\node at (1,4.1) {$\vdots$};
\node at (1,5.9) {$\vdots$};

\filldraw (2,1) circle (0.05);
\filldraw (3,2) circle (0.05);
\filldraw (4,3) circle (0.05);
\filldraw (6,5) circle (0.05);
\node[rotate=80] at (4.85,4) {$\ddots$};
\node[rotate=80] at (6.85,5.9) {$\ddots$};

\end{tikzpicture}
\end{center}
However, it is not transitively closed because
$\{(v,u),(w,u),(v,h)\}\subseteq R$ and \mbox{$(w,h)\notin R$}. We cannot remedy the situation by taking the transitive closure of $R$ because
the transitive closure
$\overline{R}=\mathrm{FP}(E)\times\mathrm{FP}(F)$ is not proper, whence it is not admissible.
\end{example}

\section{Examples and counterexamples of factorizations}\label{sec:7}
\noindent
The aim of this section is to instantiate the diversity of possible situations one encounters when trying to factorize admissible relations
into admissible graph homomorphisms and admissible path homomorphisms of graphs.
Example~3.12 already shows that not every $R\in\mathsf{RG}(E,F)$ admitting 
a pullback-type factorization enjoys a pushout-type factorization.
Below, we give two examples of $R\in\mathsf{ARG}(E,F)$ that admit admissible pullback-type factorizations 
but whose pushout-type factorizations
either fail the target-surjectivity condition or the regularity condition.

\begin{example}
Consider the following commutative diagram:
\begin{center}
\begin{tikzpicture}[xscale=7,yscale=3.8]

\node[draw,rectangle,color=gray!60] (a1) at (135:1) {\includegraphics[page=1,valign=c]{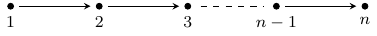}};
\node[draw,rectangle,color=gray!60] (a2) at ($(45:1)+(135:1)$) {\includegraphics[page=2,valign=c]{images.pdf}};
\node[draw,rectangle,color=gray!60] (b1) at (0,0) {\includegraphics[page=28,valign=c]{images.pdf}}; 
\node[draw,rectangle,color=gray!60] (b2) at (45:1) {\includegraphics[page=3,valign=c]{images.pdf}};

\path[-To,font=\footnotesize,inner sep=1pt,shorten <=2pt,shorten >=2pt]
(a1) edge[dashed] node[above=2pt] {$R$} (b2)
(a2) edge node[above left] {$\vartheta$} (a1)
(b2) edge node[below right] {$\vartheta'$} (b1)
(a1) edge node[below left] {$\varphi'$} (b1)
(a2) edge node[above right] {$\varphi$} (b2);

\draw
(a1.north) node[above left] {$E$}
(b2.north) node[above] {$F$}
(a2.west) node[left] {$G_R$}
(b1.west) node[left] {$H$};

\end{tikzpicture}.
\end{center}
Here the graph $E$ has vertices $\{1,\ldots,n\}$ and one edge $e_i$ from the vertex $i$ to the vertex $i+1$ for every $1\leq i\leq n-1$.
The graph $F$ has vertices $\{1,2\}$ and edges $\{f_1,\ldots,f_{n-1}\}$
all with the source $1$ and the target~$2$.
The admissible relation $R$ is 
\begin{equation}
R:=\{(i,1)\mid 1\leq i\leq n-1\}\cup\{(n,2)\}\cup\{(e_i\ldots e_{n-1},f_i)\mid 1\leq i\leq n-1\}.
\end{equation}
Next, as $F$ is row finite, we conclude  from Proposition~\ref{rowf} that 
$R=R_\varphi\circ R^\vartheta$ with $\vartheta\in\mathsf{RMIPG}(G_R,E)$ and 
$\varphi\in\mathsf{CRTBPOG}(G_R,F)$. We easily verify that
the relation graph $G_R$ is isomorphic to the graph with $n$ vertices and one edge $g_i$ 
from the  $i$th vertex to the  $n$th vertex for every $1\leq i\leq n-1$. Also, we can explicitly compute:
\begin{gather}
\vartheta\big(i\big)=i,\;1\leq i\leq n,\quad\vartheta(g_i)=e_i\ldots e_{n-1},\;1\leq i\leq n-1,
\nonumber\\
\varphi\big(1\big)=\ldots=\varphi\big(n-1\big)=1 , \quad
\varphi\big(n\big)=2 , \quad
\varphi(g_i)=f_i\;1\leq i\leq n-1 .
\end{gather}

On the other hand, it is easy to check that $R=R^{\vartheta'}\circ R_{\varphi'}$ where the codomain $H$ of $\varphi'$ and $\vartheta'$ has vertices
 $\{1,2\}$, $n-2$ loops $h_1,\ldots,h_{n-2}$ based at the vertex $1$, and one edge $h_{n-1}$ from the vertex $1$ to the vertex $2$.
The morphism $\varphi':E\to H$ is given by
\[
\varphi'(1)=\ldots=\varphi'(n-1)=1 , \qquad
\varphi'(n)=2 , \qquad
\varphi'(e_i)=h_i ,\,\;\forall\;1\leq i\leq n-1.
\]
The morphism  $\vartheta'$ sends $f_i$ to the path $h_ih_{i+1}\ldots h_{n-1}$.
It is easy to check that $\vartheta'\in
\mathsf{MIPG}(F,H)$. On the flip side, $\varphi'\in\mathsf{POG}(E,H)$ is regular but fails to be target surjective.
It is important to emphasize that there is no pushout-type factorization that can remedy the lack of target surjectivity because
$1\in H^0$ is cannot be a source and $1\in\varphi'^{-1}(1)$ is a source.

Note that the Leavitt path algebras (respectively graph C*-algebras) of the graphs
$E$, $F$ and $G_R$ are all isomorphic to $M_n(k)$ (respectively $M_n(\mathbb{C})$) \cite{R05}.
One can check that the algebra homomorphisms induced by $R$, $R_\varphi$ and $R^\vartheta$ are isomorphisms.
The graph $H$ is a $1$-sink extension of the bouquet graph with $n-2$ loops.
\end{example}

\begin{example}
Consider the following commutative diagram:
\begin{center}
\begin{tikzpicture}[xscale=5,yscale=3.8]

\node[draw,rectangle,color=gray!60] (a1) at (135:1) {\includegraphics[page=9,valign=c]{images.pdf}};
\node[draw,rectangle,color=gray!60] (a2) at ($(45:1)+(135:1)$) {\includegraphics[page=8,valign=c]{images.pdf}};
\node[draw,rectangle,color=gray!60] (b1) at (0,0) {\includegraphics[page=7,valign=c]{images.pdf}};
\node[draw,rectangle,color=gray!60] (b2) at (45:1) {\includegraphics[page=7,valign=c]{images.pdf}};

\path[-To,font=\footnotesize,inner sep=1pt,shorten <=2pt,shorten >=2pt]
(a1) edge[dashed] node[above=2pt] {$R$} (b2)
(a2) edge node[above left] {$\vartheta$} (a1)
(b2) edge node[below right] {$\vartheta'$} (b1)
(a1) edge node[below left] {$\varphi'$} (b1)
(a2) edge node[above right] {$\varphi$} (b2);

\draw
(a1.north) node[above] {$E$}
(b2.north) node[above] {$F$}
(a2.west) node[left] {$G_R$}
(b1.west) node[left] {$H$};

\end{tikzpicture}.
\end{center}
Here the graph $E$ has vertices $\{1,2\}$, a loop $e_{11}$ at the vertex $1$,
and edges $e_{12}$ from $1$ to $2$ and $e_{21}$ from $2$ to $1$. The graph
$F$ has one vertex and two loops $f_1,f_2$.
The admissible relation $R$ is generated by
\begin{equation}
G_R^0:=\{(1,1),(2,1)\} \quad\text{and}\quad
G_R^1:=\big\{(e_{11},f_1),(e_{12},f_1),(e_{21}e_{11},f_2),(e_{21}e_{12},f_2)\big\}
\end{equation}
using Lemma~\ref{lem:generation}. We find that $G_R$
is isomorphic to the graph with vertices $\{1,2\}$ and one arrow $g_{ij}$ from the vertex $i$ to the vertex $j$ for all $i,j\in\{1,2\}$.
 (It is the line graph of~$F$.) It 
follows from Proposition~\ref{rowf} that $R=R_\varphi\circ R^\vartheta$ with $\vartheta\in\mathsf{RMIPG}(G_R,E)$ and $\varphi\in\mathsf{CRTBPOG}(G_R,F)$.
Explicitly, $\varphi$ and $\vartheta$ are given by
\begin{gather}
\varphi(1)=\varphi(2)=1 , \qquad
\varphi(g_{11})=\varphi(g_{12})=f_1 , \qquad
\varphi(g_{21})=\varphi(g_{22})=f_2 , \nonumber\\
\vartheta(i)=i , \qquad
\vartheta(g_{1i})=e_{1i} , \qquad
\vartheta(g_{2i})=e_{21}e_{1i} , \qquad\forall\;i=1,2.
\end{gather}

On the other hand, it is easy to check that $R=R^{\vartheta'}\circ R_{\varphi'}$, where the codomain $H$ of $\varphi'$ and $\vartheta'$  
has one vertex and two loops $h_1,h_2$, and
the morphisms $\varphi'$ and $\vartheta'$ are given by
\begin{gather}
\varphi'(1)=\varphi'(2)=1 , \qquad
\varphi'(e_{11})=\varphi'(e_{12})=h_1 , \qquad
\varphi'(e_{21})=h_2 , \nonumber\\
\vartheta'(1)=1 , \qquad
\vartheta'(f_1)=h_1 , \qquad
\vartheta'(f_2)=h_2h_1 .
\end{gather}
Note that $\vartheta'$ is not the identity. In this example,
$\varphi'\in\mathsf{POG}(E,H)$ is regular but not target surjective, while $\vartheta'\in\mathsf{MIPG}(F,H)$ is not regular.
Moreover, one can easily check that $H$ must be a subgraph of any graph that might afford an admissible pushout-type factorization, and that the formulas for
$\vartheta'$ and $\varphi'$ are the only possible formulas. Hence,  an admissible pushout-type factorization of $R$ does not exist.

Finally, observe that the graph C*-algebras of all four graphs are isomorphic to the Cuntz algebra~$\mathcal{O}_2$, 
and that the maps $(R^\vartheta)^*$, $(R_\varphi)^*$ 
and $R^*$ are known to be isomorphisms.
On the other hand, $R^{\vartheta'}$ and $R_{\varphi'}$ do not induce homomorphisms of Levitt path
algebras, or *-homomorphisms of graph C*-algebras.
\end{example}

Our next example shows that the  assumption of row-finiteness in Proposition~\ref{rowf} is not redundant. It also shows that there might exist an admissible 
pushout-type factorization where there is no admissible pullback-type factorization.

\begin{example}\label{ex:nopull}
Consider the following commutative diagram
\begin{center}
\begin{tikzpicture}[xscale=5,yscale=3]

\node[draw,rectangle,color=gray!60] (a1) at (135:1) {\includegraphics[page=14,valign=c]{images.pdf}};
\node[draw,rectangle,color=gray!60] (a2) at ($(45:1)+(135:1)$) {\includegraphics[page=16,valign=c]{images.pdf}};
\node[draw,rectangle,color=gray!60] (b1) at (0,0) {\includegraphics[page=13,valign=c]{images.pdf}};
\node[draw,rectangle,color=gray!60] (b2) at (45:1) {\includegraphics[page=15,valign=c]{images.pdf}};

\path[-To,font=\footnotesize,inner sep=1pt,shorten <=2pt,shorten >=2pt]
(a1) edge[dashed] node[above=2pt] {$R$} (b2)
(a2) edge node[above left] {$\vartheta$} (a1)
(b2) edge node[below right] {$\vartheta'$} (b1)
(a1) edge node[below left] {$\varphi'$} (b1)
(a2) edge node[above right] {$\varphi$} (b2);

\draw
(a1.north) node[above] {$E$}
(b2.north) node[above] {$F$}
(a2.west) node[left] {$G_R$}
(b1.west) node[left] {$H$};

\end{tikzpicture}.
\end{center}
Here $E$ has vertices $1,\,2,\,3$, edges $e_{12},\,e_{13}$, and the source and target maps given by $s_E(e_{ij}):=i$ and $t_E(e_{ij}):=j$
for all~$e_{i,j}\in E^1$.
The graph $F$ has vertices $1,\,2,\,4$, an edge $e_{12}$ from $1$ to $2$ and countably infinitely many edges from the vertex $1$ 
to the vertex~$4$. The admissible relation $R$ equals $G_R^0  \cup G_R^1$ with
$G_R^0:=\{(1,1),(2,2)\}$ and $G_R^1:=\{(e_{12},e_{12})\}$.
Now, we find that $G_R$
is isomorphic to the graph with vertices $1,\,2$ and one arrow $e_{12}$ from the vertex $1$ to the vertex~$2$.
The morphisms $\varphi$ and $\vartheta$ are the obvious inclusions.
In this example, Proposition~\ref{rowf} does not apply because $F$ is not row finite,
and one can check that its claim does not hold. Indeed, $\varphi\in\mathsf{CRTBPOG}(G_R,F)$ and $\vartheta\in\mathsf{MIPG}(G_R,E)$, but the latter
 is not regular.
Since the pullback-type factorization in the category $\mathsf{CRG}$ exists and is uniquely given by the diagram above, this proves there exists no  
admissible
pullback-type factorization of~$R$.

On the other hand, it is easy to check that $R=R^{\vartheta'}\circ R_{\varphi'}$, where the codomain graph $H$ of $\varphi'$ and $\vartheta'$  
has vertices $1,\,2,\,3,\,4$, one edge $e_{12}$ from $1$ to $2$, one edge $e_{13}$ from $1$ to $3$ and countably infinitely many edges
from $1$ to $4$.
The morphisms $\varphi'$ and $\vartheta'$ are, again, the obvious inclusions of subgraphs. One easily checks that 
$\varphi'\in\mathsf{CRTBPOG}(E,H)$ and $\vartheta'\in\mathsf{RMIPG}(F,H)$.
\end{example}

Next, let us consider a family of examples motivated by noncommutative topology, involving quantum spheres and quantum balls. 
In these examples, both the admissible pullback-type factorizations and admissible pushout-type  factorizations exist.

\begin{figure}[t]
\begin{tikzpicture}[xscale=7.5,yscale=6]

\node[draw,rectangle,color=gray!60] (a1) at (135:1) {\includegraphics[page=4,valign=c]{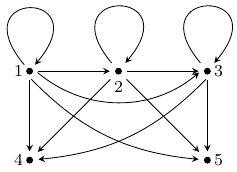}};
\node[draw,rectangle,color=gray!60] (a2) at ($(45:1)+(135:1)$) {\includegraphics[page=10,valign=c]{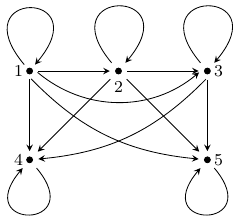}};
\node[draw,rectangle,color=gray!60] (b1) at (0,0) {\includegraphics[page=5,valign=c]{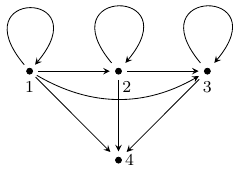}};
\node[draw,rectangle,color=gray!60] (b2) at (45:1) {\includegraphics[page=6,valign=c]{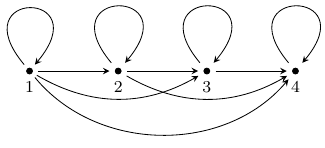}};

\path[-To,font=\footnotesize,inner sep=1pt,shorten <=2pt,shorten >=2pt]
(a1) edge[dashed] node[above=2pt] {$R$} (b2)
(a2) edge node[above left] {$\vartheta$} (a1)
(b2) edge node[below right] {$\vartheta'$} (b1)
(a1) edge node[below left] {$\varphi'$} (b1)
(a2) edge node[above right] {$\varphi$} (b2);

\draw
(a1.north) node[above] {$E$}
(b2.north) node[above] {$F$}
(a2.west) node[left] {$G_R$}
(b1.west) node[left] {$H$};

\end{tikzpicture}
\caption{$S^{6}_q\to S^{7}_q$}
\label{fig:sfere}

\vspace*{5pt}
\end{figure}

\begin{example}\label{qsph}
Let $n\geq 1$. Consider the following graphs $E$ and $F$:
\begin{align}
E^0 &:=\{1,\ldots,n+2\}  \nonumber\\
E^1 &:=\{e_{i,j}\mid 1\leq i\leq j\leq n\}\cup \{e_{i,n+1},e_{i,n+2}\mid 1\leq i \leq n\} ,  \nonumber\\
F^0 &:=\{1,\ldots,n+1\}  \nonumber\\
F^1 &:=\{e_{i,j}\mid 1\leq i\leq j\leq n+1\} .
\end{align}
For both graphs, the source and target maps are given by $s(e_{i,j})=i$ and $t(e_{i,j})=j$ for all $i,j$.
The admissible relation $R$ is generated by
\begin{align}
G_R^0 &:=\big\{(i,i)\mid 1\leq i\leq n+1\big\}\cup\big\{(n+2,n+1)\big\} , \nonumber\\
G_R^1 &:=\big\{(e_{i,j},e_{i,j})\mid 1\leq i\leq j\leq n\big\}\cup\big\{(e_{i,n+1},e_{i,n+1}) ,(e_{i,n+2},e_{i,n+1}) \mid 1\leq i\leq n\big\}   \nonumber\\
&\qquad\cup\big\{(n+1,e_{n+1,n+1}),(n+2,e_{n+1,n+1})\big\} ,
\end{align}
using Lemma~\ref{lem:generation}.
It is straightforward to verify that $G_R$
is isomorphic to the graph obtained from $E$ by adding two loops $e_{n+1,n+1}$ and $e_{n+2,n+2}$ and the vertices $n+1$
 and $n+2$, respectively.
The morphism $\varphi$ identifies the vertex $n+2$ with $n+1$ and the edge $e_{i,n+2}$ with $e_{i,n+1}$ for all $i$.
The morphism $\vartheta$ is the identity on all vertices and edges except $e_{n+1,n+1}$ and $e_{n+2,n+2}$, and maps
$e_{n+1,n+1}$ to the vertex $n+1$ and $e_{n+2,n+2}$ to the vertex $n+2$.
It follows from Proposition~\ref{rowf} that $R=R_\varphi\circ R^\vartheta$ with $\vartheta\in\mathsf{RMIPG}(G_R,E)$ and 
$\varphi\in\mathsf{CRTBPOG}(G_R,F)$.

On the other hand, it is easy to check that $R=R^{\vartheta'}\circ R_{\varphi'}$, where the codomain graph $H$ of $\varphi'$ and $\vartheta'$  
has $H^0:=F^0$ and $H^1:=F^1\setminus\{e_{n+1,n+1}\}$.
The morphism $\vartheta'$ is  given by the identity on the part of $F$ coinciding with $H$ and maps $e_{n+1,n+1}$ to $n+1$.
The morphism $\varphi'$ identifies the vertex $n+2$ with $n+1$ and the edge $e_{i,n+2}$ with $e_{i,n+1}$ for all~$i$.
One verifies that $\vartheta'\in\mathsf{RMIPG}(F,H)$ and $\varphi'\in\mathsf{CRTBPOG}(E,H)$.
 
To end with, note that $C^*(E)\cong C(S^{2n}_q)$,
$C^*(F)\cong C(S^{2n+1}_q)$ and
$C^*(H)\cong C(B^{2n}_q)$~\cite{HS02}.
For $n=1$, $C^*(G_R)$ is the C*-algebra of a quantum lens space~\cite{bs18}.
For $n=3$, the commutative diagram is in Figure~\ref{fig:sfere}.
\end{example}

Our next family of examples concerns unital embeddings \cite{k-k09,hl24} of the celebrated Cuntz algebras~\cite{c-j77}.

\begin{example}\label{ex:cuntz}
Let $m,n,k$ be positive integers with $n\geq 2$, and  such that $m-1=k(n-1)$. Consider the following commutative diagram of admissible
morphisms in Figure~\ref{fig:Cuntz}.

\begin{figure}[t]
\begin{tikzpicture}[xscale=6,yscale=6]

\node[draw,rectangle,color=gray!60] (a1) at (135:1) {\includegraphics[page=36,valign=c]{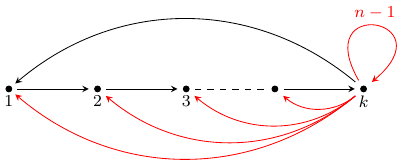}};
\node[draw,rectangle,color=gray!60] (a2) at ($(45:1)+(135:1)$) {\includegraphics[page=37,valign=c]{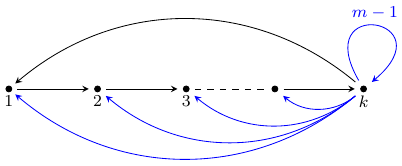}};
\node[draw,rectangle,color=gray!60] (b1) at (0,0) {\includegraphics[page=34,valign=c]{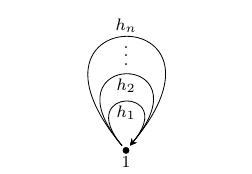}};
\node[draw,rectangle,color=gray!60] (b2) at (45:1) {\includegraphics[page=35,valign=c]{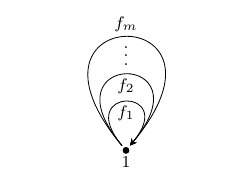}};

\path[-To,font=\footnotesize,inner sep=1pt,shorten <=2pt,shorten >=2pt]
(a1) edge[dashed] node[above=2pt] {$R$} (b2)
(a2) edge node[above left] {$\vartheta$} (a1)
(b2) edge node[below right] {$\vartheta'$} (b1)
(a1) edge node[below left] {$\varphi'$} (b1)
(a2) edge node[above right] {$\varphi$} (b2);

\draw
(a1.north) node[above] {$E$}
(b2.north) node[above] {$F$}
(a2.west) node[left] {$G_R$}
(b1.west) node[left] {$H$};

\end{tikzpicture}
\caption{}
\label{fig:Cuntz}
\end{figure}

Here $E,F,H$ are the graphs given by
\begin{align}
	&E^0 :=\{1,\ldots,k\},\quad
	E^1 :=\{e_{a,b}\mid 0\leq a\leq n-1,\, 1\leq b\leq k\}, \nonumber\\
	&t_E(e_{a,b})=b,\quad
	 s_E(e_{a,b})=\begin{cases}
	 b-1 \text{ for } a=0,\, b\neq 1 \\
	 k \text{ for } a=0\text{ and }b=1, \text{ or } a\neq 0,
	 \end{cases}
	\nonumber\\
	&F^0 :=\{1\}, \quad
	F^1 :=\{f_{j}\mid 1\leq j\leq m\}, \nonumber\\
	&H^0 :=\{1\},\quad
	H^1 :=\{h_{i}\mid 1\leq i\leq n\}.
\end{align}
Now, consider the morphisms $\varphi'\in \mathsf{CRTBPOG}(E,H)$ and $\vartheta'\in\mathsf{RMIPG}(F,H)$ given by:
\begin{align}
\varphi'(e_{a,b})&=h_{a+1},\; 0\leq a\leq n-1,\;1\leq b\leq k, \nonumber\\
	\vartheta'(f_{(n-1)l+r})&=h_n^lh_r,\; 0\leq l\leq k-1,\ 1\leq r\leq n-1,\nonumber\\
	\vartheta'(f_m)&=h_n^k.
\end{align}
Define an admissible relation $R:=R^{\vartheta'}\circ R_{\varphi'}$,  and compute its pullback-type factorization graph:
 $G_R^0=\{(i,1)\mid 1\leq i\leq k\}\cong\{1,\ldots, k\}$,
\begin{align}
	G_R^1=\, &\{(e_{0,b},f_1)\mid 1\leq b\leq k\}\nonumber\\
	\cup\, &\{(e_{n-1,k}^{l-1}\,e_{n-1,c}\,e_{0,c+1},f_{(n-1)l+1})\mid 1\leq c\leq k-1,\ 1\leq l\leq k-1\}\nonumber\\
	\cup\, &\{(e_{n-1,k}^{l}\,e_{0,1},f_{(n-1)l+1})\mid 1\leq l\leq k-1\}\nonumber\\
	\cup\, &\{(e_{n-1,k}^l\,e_{r-1,c}\,,\,f_{(n-1)l+r})\mid 1\leq c\leq k,\ 0\leq l\leq k-1,\ 2\leq r\leq n-1\}\nonumber\\
	\cup\, &\{(e_{n-1,k}^{k-1}\,e_{n-1,c}\,,\,f_m)\mid 1\leq c\leq k\},\nonumber\\
s_{G_R}(x,y)=\, &s_E(x),\quad t_{G_R}(x,y)=t_E(x).
\end{align}
Note that there is exactly one edge going from $i$ to $i+1$ for $1\leq i\leq k-1$,
there are $m-1$ edges going from $k$ to $i$ for $2\leq i\leq k$, and there are $m$ edges going from $k$ to~$1$, which is depicted
at the top of our diagram of graphs.
The morphisms $\varphi$ and $\vartheta$ given by respective projections 
are admissible by Proposition~6.11 because $F$ is row finite.
Combined with the above given morphisms $\varphi'$ and~$\vartheta'$, they  yield a commutative diagram
of graphs providing both the pullback-type  and pushout-type admissible factorizations.
As explained in~\cite{hl24}, the above commutative diagram of graphs yields the following commutative diagram of their graph C*-algebras:
\begin{center}
\begin{tikzpicture}[scale=2.3,baseline=(current bounding box.center)]

\node (a1) at (135:1) {$M_k(\mathcal{O}_{m})$};
\node (a2) at ($(45:1)+(135:1)$) {$M_k(\mathcal{O}_{k(m-1)+1})$};
\node (b1) at (0,0) {$\phantom{.}\mathcal{O}_n.$};
\node (b2) at (45:1) {$\mathcal{O}_m$};

\path[To-,font=\footnotesize,inner sep=2pt,shorten <=2pt,shorten >=2pt]
(a1) edge node[above=2pt] {$R^*$} (b2)
(a1) edge node[above left] {$\vartheta_*$} (a2)
(b1) edge node[below right] {$\vartheta'_*$} (b2)
(a1) edge node[below left] {$\varphi'^*$} (b1)
(a2) edge node[above right] {$\varphi^*$} (b2);

\end{tikzpicture}
\end{center}
\vspace*{-9mm}
\end{example}

The following example shows that an admissible relation can admit neither admissible pullback-type nor admissible pushot-type factorizations.


\begin{example}\label{ex:last}
Consider the following relation morphism:
\begin{center}
\begin{tikzpicture}[scale=6]

\node[draw,rectangle,color=gray!60] (a1) at (0,0) {\includegraphics[page=18,valign=c]{images.pdf}};
\node[draw,rectangle,color=gray!60] (b2) at (1,0) {\includegraphics[page=17,valign=c]{images.pdf}};

\path[-To,font=\footnotesize,inner sep=1pt,shorten <=2pt,shorten >=2pt]
(a1) edge node[above=2pt] {$R$} (b2);

\draw
(a1.west) node[left] {$E$}
(b2.east) node[right] {$F$};

\end{tikzpicture}.
\end{center}
Here $E$ is given by $E^0:=\{1,2,3,5,6\}$ and $E^1:=\{e_{12},e_{13},e_{56}\}$ with $s_E(e_{ij}):=i$ and $t_E(e_{ij}):=j$ for all $e_{ij}\in E^1$.
The graph $F$ has vertices $1,\,2,\,4,\,5$, one edge $f_{12}$ from $1$ to~$2$, and countably infinitely many edges from $1$ to~$4$.
As there are no paths longer than one here, the  admissible relation $R$ is the union of 
\begin{equation}
G_R^0:=\{(1,1),\,(2,2),\,(5,5),\,(6,5)\} \quad\text{and}\quad
G_R^1:=\{(e_{12},f_{12})\} .
\end{equation}
Arguing as in Example~\ref{ex:nopull}, we can easily see that $R$ has no admissible pullback-type factorization.
Furthermore, neither does $R$ admit an admissible pushout-type factorization. Indeed, suppose that such a factorization exists. Then the graph homomorphism 
from $E$  would identify vertices $5$ and $6$ mapping the edge $e_{56}$ to a loop, which would contradict the target-bijectivity
 condition.
\end{example}

It turns out that we can remedy the above lack of admissible factorizations by allowing more factorization steps. This brings us to our last example.
\begin{example}\label{fdec}
Let $R\in\mathsf{ARG}(E,F)$ be the relation in Example~\ref{ex:last}. Then one has the following admissible factorization:
\begin{center}
\begin{tikzpicture}[scale=4]

\node[draw,rectangle,color=gray!60] (a) at (0,0) {\includegraphics[page=18,valign=c]{images.pdf}};
\node[draw,rectangle,color=gray!60] (b) at (1,0) {\includegraphics[page=19,valign=c]{images.pdf}};
\node[draw,rectangle,color=gray!60] (c) at (2,0) {\includegraphics[page=20,valign=c]{images.pdf}};
\node[draw,rectangle,color=gray!60] (d) at (3,0) {\includegraphics[page=17,valign=c]{images.pdf}};

\path[-To,font=\footnotesize,inner sep=1pt,shorten <=2pt,shorten >=2pt]
(a) edge node[above=2pt] {$R^\vartheta$} (b)
(b) edge node[above=2pt] {$R_\varphi$} (c)
(c) edge node[above=2pt] {$R^{\vartheta'}$} (d);

\draw
(a.south east) node[right] {$E$}
(b.south east) node[right] {$G$}
(c.south east) node[right] {$H$}
(d.south) node[below] {$F$};

\end{tikzpicture}.
\end{center}
Here $\vartheta$ and $\vartheta'$ are the obvious inclusions, and $\varphi$ embeds the top component of the graph and sends both vertices $5$ and $6$ to 
the vertex~$5$. 
One easily checks that
$\vartheta\in\mathsf{RMIPG}(G,E)$, $\vartheta'\in\mathsf{RMIPG}(F,H)$ and $\varphi\in\mathsf{CRTBPOG}(G,H)$.
\end{example}

\section*{Acknowledgments}
\noindent
This research is part of the EU Staff Exchange project 101086394 \emph{Operator Algebras That One Can See}.
The project is co-financed by the Polish Ministry of Education and Science under the program PMW (grant agreement 5305/HE/2023/2).
 G.\ G.~de Castro and F.~D'Andrea hereby acknowledge the financial support of their research visits to Warsaw by the University of Warsaw 
 Thematic Research Programme \emph{Quantum Symmetries}. They are also happy to thank IMPAN for the hospitality. G.\ G.~de Castro was also 
 partially supported by Fundacao de Amparo a Pesquisa e Inovacao do Estado de Santa Catarina (FAPESC), Edital 21/2024 and Conselho Nacional de Desenvolvimento Científico e Tecnológico (CNPq) - Brazil.
P.\ M.~Hajac and F.~D'Andrea are grateful to the Universidade Federal de Santa Catarina in Florian\'opolis for the hospitality.

\end{document}